\documentclass[12pt,a4paper]{amsart}  
\usepackage{graphicx,amssymb,stmaryrd,xspace,color}
\RequirePackage[raiselinks=false,colorlinks=true,citecolor=blue,urlcolor=black,linkcolor=blue,bookmarksopen=true]{hyperref}
\newcommand{\arxiv}[1]{\href{http://arxiv.org/pdf/#1}{arXiv:#1}}
\newcommand{\doi}[1]{\href{http://doi.org/#1}{DOI:#1}}

\usepackage{eucal}

\DeclareMathAlphabet{\mathbbe}{U}{bbold}{m}{n}
\newcommand{\simplexcategory}{\mathbbe{\Delta}}
\newcommand{\newnabla}{\rotatebox[origin=c]{180}{\(\simplexcategory\)}}

\usepackage[margin=25mm]{geometry}			     

\newcommand{\dstrees}{\mathbf{H}}
\newcommand{\intconcat}{\vee}
\newcommand{\terminal}{1}

\usepackage{texdraw}
\input{txdtools}

\everytexdraw{%
	\drawdim pt \linewd 0.5 \textref h:C v:C
	\setunitscale 1
  \arrowheadtype t:F
  \arrowheadsize l:5 w:3
}
\newcommand{\Onedot}{
  \bsegment
    \move (0 0) \fcir f:0 r:2.5
  \esegment
}

\newcommand{\onedot}{
  \bsegment
    \move (0 0) \fcir f:0 r:2
  \esegment
}
  
\newcommand{\twoVtr}{
  \bsegment
    \move (0 0) \onedot 
    \lvec (-7 15) 
    \move (0 0)
    \lvec (7 15) 
  \esegment
}
\newcommand{\lowercut}{
  \bsegment
    \move (-17 4) \clvec (-7 9)(7 9)(17 4)
  \esegment
} 

\newcommand{\twolintr}{
  \bsegment
    \move (0 0) \onedot 
    \lvec (0 15) \onedot
    \lvec (0 30) \onedot
  \esegment
}
 
\newcommand{\twoskewtr}{
  \bsegment
    \move (0 0) \onedot 
    \lvec (-13 8) \onedot
    \move (0 0)
    \lvec (0 21) \onedot
  \esegment
} 

\newcommand{\uppercut}{
  \bsegment
    \move (-25 18) \clvec (-7 9)(7 9)(20 1)
  \esegment
} 

\usepackage[v2,all]{xy}
\newcommand{\ulpullback}[1][ul]{\save*!/#1-4ex/#1:(-1,1)@^{|-}\restore}
\newcommand{\dlpullback}[1][dl]{\save*!/#1-4ex/#1:(-1,1)@^{|-}\restore}
\newcommand{\urpullback}[1][ur]{\save*!/#1-4ex/#1:(-1,1)@^{|-}\restore}
\newcommand{\drpullback}[1][dr]{\save*!/#1-4ex/#1:(-1,1)@^{|-}\restore}

\usepackage{ stmaryrd }
\newcommand{\genmap}{\rightarrow\Mapsfromchar}

\def\eq{{\mathrm{eq}}}

\def\Map{{\mathrm{Map}}}

\def\M{M\"obius\xspace}

\newdir{ >}{{}*!/-7pt/@{>}}
  
\def\overarrow#1{{\vec{#1}}}
\def\nondeg{\overarrow}

\newcommand{\Phieven}{\Phi_{\text{\rm even}}}
\newcommand{\Phiodd}{\Phi_{\text{\rm odd}}}

\newcommand{\ov}{\overline}
\newcommand{\un}{\underline}
\newcommand{\unDelta}{\un\simplexcategory}
\newcommand{\Deltaact}{\simplexcategory_{\text{\rm act}}}

\newcommand{\culf}{CULF\xspace}
\newcommand{\culfsymb}{\mathrm{culf}}

\newcommand{\Decomp}{\kat{Dcmp}}

\newcommand{\vect}{\kat{vect}}
\newcommand{\Decbot}[1]{\operatorname{Dec}_\bot{}\kern-2pt{#1}}
\newcommand{\Dectop}[1]{\operatorname{Dec}_\top{}\kern-2pt{#1}}

\providecommand{\norm}[1]{\left| {#1}\right|}

\def\onto{\twoheadrightarrow}
\def\into{\hookrightarrow}
\newcommand{\shortsetminus}{\,\raisebox{1pt}{\ensuremath{\mathbb r}\,}}

\providecommand{\kat}[1]{\text{\textbf{\textsl{#1}}}}

\newcommand{\LIN}{\kat{LIN}}

\newcommand{\rat}{\rightarrowtail}

\newcommand{\upperstar}{^{\raisebox{-0.25ex}[0ex][0ex]{\(\ast\)}}}

\newcommand{\lowershriek}{_!}
\newcommand{\uppershriek}{^!}

\newcommand{\isopil}{\stackrel{\raisebox{0.1ex}[0ex][0ex]{\(\sim\)}}%
			{\raisebox{-0.15ex}[0.28ex]{\(\rightarrow\)}}}
\newcommand{\tensor}{\otimes}
\newcommand{\op}{^{\text{{\rm{op}}}}}

\newcommand{\Set}{\kat{Set}}

\newcommand{\Grpd}{{\mathcal{S}}}

\newcommand{\oneGrpd}{\kat{Grpd}}
\newcommand{\onegrpd}{\kat{grpd}}

\def\Sdot{{\ensuremath{S_\bullet}}}

\newcommand{\fatnerve}{\mathbf{N}}

\newcommand{\N}{\mathbb{N}}

\newcommand{\F}{\mathbb{F}}
\newcommand{\B}{\mathbb{B}}
\newcommand{\Q}{\mathbb{Q}}

\newcommand{\ground}{\Bbbk}

\newcommand{\MM}{\mathcal{M}}
\newcommand{\CC}{\mathcal{C}}
\newcommand{\DD}{\mathcal{D}}
\newcommand{\EE}{\mathcal{E}}

\newcommand{\name}[1]{\ulcorner #1\urcorner}

\newcommand{\Hom}{\operatorname{Hom}}
\newcommand{\Ext}{\operatorname{Ext}}
\newcommand{\Fun}{\operatorname{Fun}}
\newcommand{\PsFun}{\kat{PsFun}}
\newcommand{\Aut}{\operatorname{Aut}}
\newcommand{\id}{\operatorname{id}}

\newcommand{\Ar}{\operatorname{Ar}}

\makeatletter
\def\subsection{\@startsection{subsection}{2}%
  \z@{.5\linespacing\@plus.7\linespacing}{.5\linespacing}%
  {\normalfont\bfseries}}
\makeatother

\newtheorem{lemma}{Lemma}[section]
\newtheorem{prop}[lemma]{Proposition}
\newtheorem{thm}[lemma]{Theorem}
\newtheorem{theorem}[lemma]{Theorem}
\newtheorem{cor}[lemma]{Corollary}
\theoremstyle{definition}

\newtheorem{taller}[lemma]{$\!\!$}

\newenvironment{blanko}[1]%
{\begin{taller}{\normalfont\bfseries  #1}\normalfont}%
{\end{taller}}

\newenvironment{blanko*}[1]{\begin{list}{\bf {#1} }%
{\setlength{\labelsep}{0mm}\setlength{\leftmargin}{0mm}%
\setlength{\labelwidth}{0mm}\setlength{\listparindent}{\parindent}%
\setlength{\parsep}{\parskip}\setlength{\partopsep}{0mm}}%
\item%
}{\end{list}}

\newenvironment{deff}%
{\begin{list}{\em Definition. }%
{\setlength{\labelsep}{0mm}\setlength{\leftmargin}{0mm}%
\setlength{\labelwidth}{0mm}\setlength{\listparindent}{\parindent}%
\setlength{\parsep}{\parskip}\setlength{\partopsep}{0mm}}%
\item}{\end{list}}

\newenvironment{proof*}[1]{\begin{list}{\em #1 }%
{\setlength{\labelsep}{0mm}\setlength{\leftmargin}{0mm}%
\setlength{\labelwidth}{0mm}\setlength{\listparindent}{\parindent}%
\setlength{\parsep}{\parskip}\setlength{\partopsep}{0mm}}%
\item}{\qed\end{list}}

\thanks{%
  The first author 
  was partially supported by grants 
  MTM2012-38122-C03-01,  
  MTM2013-42178-P,       
  2014-SGR-634,          
  MTM2015-69135-P,   
  MTM2016-76453-C2-2-P  (AEI/FEDER, UE), and   
  2017-SGR-932,          
  the second author 
  by 
  MTM2013-42293-P, 
  MTM2016-80439-P  (AEI/FEDER, UE),
  and
  2017-SGR-1725,
  and the third author   
  by
  MTM2013-42178-P and
  MTM2016-76453-C2-2-P (AEI/FEDER, UE)}

\author{Imma G\'alvez-Carrillo}
\address{Departament de Matem\`atiques
      \\Universitat Polit\`ecnica de Catalunya
	  }
\email{m.immaculada.galvez@upc.edu}

\author{Joachim Kock}
\address{Departament de Matem\`atiques
       \\Universitat Aut\`onoma de Barcelona
	   }
\email{kock@mat.uab.cat}

\author{Andrew Tonks}
\address{Department of Mathematics\\ 
University of Leicester
}
\email{apt12@le.ac.uk}

\title[Decomposition spaces and incidence coalgebras]{Decomposition spaces, incidence algebras and M\"obius inversion
I: basic theory}
\date{}

\begin{document}

\begin{abstract}
  This is the first in a series of papers devoted to the theory of decomposition
  spaces, a general framework for incidence algebras and \M inversion, where
  algebraic identities are realised by taking homotopy cardinality of
  equivalences of $\infty$-groupoids.  A decomposition space is a simplicial
  $\infty$-groupoid satisfying an exactness condition, weaker than the Segal
  condition, expressed in terms of active and inert maps in $\simplexcategory$.
  Just as the Segal condition expresses composition, the new exactness condition
  expresses decomposition, and there is an abundance of examples in
  combinatorics.
  
  After establishing some basic properties of decomposition spaces, the main
  result of this first paper shows that to any decomposition space there is an
  associated incidence coalgebra, spanned by the space of $1$-simplices, and
  with coefficients in $\infty$-groupoids.  We take a functorial viewpoint
  throughout, emphasising conservative ULF functors; these induce coalgebra
  homomorphisms.  Reduction procedures in the classical theory of incidence
  coalgebras are examples of this notion, and many are examples of decalage of
  decomposition spaces.  An interesting class of examples of decomposition
  spaces beyond Segal spaces is provided by Hall algebras: the Waldhausen
  \Sdot-construction of an abelian (or stable infinity) category is shown to be
  a decomposition space.

  In the second paper in this series we impose further conditions on
  decomposition spaces, to obtain a general \M inversion principle, and to ensure
  that the various constructions and results admit a homotopy cardinality.  In
  the third paper we show that the Lawvere--Menni Hopf algebra of \M intervals
  is the homotopy cardinality of a certain universal decomposition space.  Two
  further sequel papers deal with numerous examples from combinatorics.
  
  Note: The notion of decomposition space was arrived at independently by
  Dyckerhoff and Kapranov~\cite{Dyckerhoff-Kapranov:1212.3563} who call them
  unital $2$-Segal spaces.  Our theory is quite orthogonal to theirs: the
  definitions are different in spirit and appearance, and the theories differ in
  terms of motivation, examples, and directions.
\end{abstract}



\vspace*{-8pt}

\maketitle

\small

\vspace*{-14pt}

\tableofcontents

\vspace*{-14pt}

\noindent
Note.  This paper originally formed Sections 1 and 2 of the manuscript 
\cite{GKT:1404.3202}, which has now been split into six papers.

\normalsize

\setcounter{section}{-1}

\addtocontents{toc}{\protect\setcounter{tocdepth}{1}}

\section{Introduction}

The notion of incidence algebra of a locally finite poset is an important
construction in algebraic combinatorics, with applications to many fields
of mathematics. 
In this work we generalise this construction in three
directions:
(1) we replace posets by categories and $\infty$-categories;
(2) we replace scalar coefficients in a field by $\infty$-groupoids, working at the objective level, ensuring natively bijective proofs~\cite{GKT:HLA};
and most importantly:
(3) we replace the Segal condition, which essentially characterises $\infty$-categories 
among simplicial $\infty$-groupoids, by a weaker condition that still allows
the construction of incidence algebras.
Simplicial $\infty$-groupoids satisfying this axiom
are called decomposition spaces, seen as a systematic
framework for decomposing
structures, whereas categories constitute the
systematic
framework for
composing structures.

In the present work we focus on incidence \emph{coalgebras};
incidence algebras are just the convolution algebras given by their linear duals.
The fundamental role played by coalgebras was established by Rota
and his collaborators, the work with Joni~\cite{JoniRotaMR544721} being a
milestone.

We briefly preface the historically motivated introduction below with
a preview of
one of the examples that motivated us, and
which will serve as a running example.

\begin{blanko}{Running example: the  Hopf
    algebra of rooted trees.} (We return to this example in \ref{ex:trees-decomp},
\ref{ex:trees-coalg},
\ref{ex:trees-bialg},
\ref{ex:CK}.)
\label{ex:preintro}
The Butcher--Connes--Kreimer Hopf algebra of rooted trees \cite{Butcher:1972,Dur:1986,Kreimer:9707029}
is the free commutative
  algebra on the set of iso-classes of rooted trees $T$,
  with the comultiplication defined by summing over certain admissible cuts $c$:
  \begin{equation}\label{eq:preintro}
  \Delta(T) \ = \sum_{c\in \operatorname{adm.cuts}(T)} P_c \tensor R_c
  \end{equation}
  An admissible cut $c$ partitions the nodes of $T$ into two subsets or `layers'
  \begin{eqnarray}\nonumber\\\label{eq:preintro2}
    \\[-1cm]
  \begin{texdraw}
      \setunitscale 0.7
      \bsegment
	\move (0 0) \twoVtr
	\move (-7 15) \twoskewtr
	\move (7 15) \twolintr
	\move (0 15)
	\uppercut
	\htext (-25 0){\footnotesize $R_c$}
	\htext (30 30){\footnotesize $P_c$}
      \esegment
  \end{texdraw}\nonumber
  \end{eqnarray}
 One layer must form a rooted subtree $R_c$ (or be empty), and its complement forms the `crown', a subforest $P_c$ regarded as a monomial of trees.

  We can formalise this construction as follows.
  Let $\dstrees_k$ denote the groupoid of
  forests with $k-1$ compatible admissible cuts, partitioning the forest into $k$ layers (which may be empty). 
 These form a simplicial groupoid $\dstrees$, where simplicial degeneracy maps repeat a cut, inserting an empty layer, and
   face maps forget a cut, joining adjacent layers, or discard
   the top or bottom layer.

  The comultiplication \eqref{eq:preintro} arises from this simplicial groupoid by a pull-push 
  formula (see \ref{comult}, \ref{ex:trees-coalg} below): for a tree $T\in \dstrees_1$, take the homotopy sum over the
  homotopy fibre $d_1^{-1}(T) \subset \dstrees_2$, and for each element $c$ in the fibre return
  the pair $(d_2 c, d_0 c)$ consisting of the two layers.  Finally take homotopy 
  cardinality to arrive at $P_c \tensor R_c$.

  We note
  three things about this construction.
  Firstly, it is essential to work with simplicial
  groupoids rather than simplicial sets: had we passed to sets of
  iso-classes (of forests with cuts), crucial information would be lost, essentially
  because trees with a cut admit isomorphisms that do not fix the
  underlying tree --- see \cite{GKT:ex} for detailed explanation of this
  point.
  Secondly, we took homotopy cardinality at the last step, but in fact the whole
  construction is so formal and natural that it works on the `objective level'
  of groupoid slices.  Refraining from taking cardinality yields a natively
  `bijective' version.  Finally, and most importantly, this simplicial groupoid
  is not a Segal object: that is, is not the (fat) nerve of a category.
  Indeed, the Segal condition would imply that any tree with an
  admissible cut could be reconstructed uniquely knowing just the layers above
  and below the cut.  But this is manifestly false: there are \emph{many} trees
  with cuts that have the same layers as \eqref{eq:preintro2}.
  
  The main discovery is that there is
  a weaker condition than the
  Segal condition that allows the construction of a coassociative incidence 
  coalgebra: this is the decomposition-space axiom that we introduce in \S\ref{sec:decomp}.  It has a clear combinatorial interpretation (see the pictures in  \ref{ex:trees-decomp}),
  has a clean categorical description as an 
  exactness condition~\ref{def:decomp}, and is a general condition satisfied also by  examples from other areas of mathematics, such as  the Waldhausen $\Sdot$-construction~\ref{Sdot}. See  Dyckerhoff and  Kapranov~\cite{Dyckerhoff-Kapranov:1212.3563} for further outlook.
\end{blanko}

\subsection*{Background and motivation}

Leroux's notion of \M category \cite{Leroux:1975} generalises at the same 
time locally finite posets (Rota~\cite{Rota:Moebius}) and Cartier--Foata
finite-decomposition monoids~\cite{Cartier-Foata}, the two classical settings 
for incidence (co)algebras and \M inversion.
The finiteness conditions in the definition of \M category ensure that the 
comultiplication law 
\begin{equation}\label{Delta-f}\Delta(f)\:=\sum_{b \circ a=f}a\otimes b \end{equation}
is well defined on the vector space spanned by the set of arrows.
This defines the classical incidence coalgebra.

An important advantage of having the classical settings of posets and monoids
on the same footing is
they may then be connected by an appropriate class of functors, the 
{\em \culf} functors
(standing for `conservative'
and `ULF' = `unique lifting of factorisations'; see \S\ref{sec:cULF}).
In particular it gives a nice explanation of the important process of reduction,
to get the most interesting algebras out of posets,
a process that was sometimes rather ad hoc.
For the most classical example of this process,
consider the divisibility poset $(\N^\times,\mid)$ as a category.
It admits a \culf functor to the multiplicative monoid $(\N^\times,\times)$,
considered as a category with only one object.
This functor induces a homomorphism of incidence coalgebras
which is precisely the reduction map from the `raw' incidence coalgebra of the
divisibility poset to its reduced incidence coalgebra, which is isomorphic to the
Cartier--Foata incidence coalgebra of the multiplicative monoid.

Shortly after Leroux's work, D\"ur~\cite{Dur:1986} studied more involved categorical
structures to extract further examples of incidence 
algebras and study their \M functions.
In particular he realised
the Hopf algebra of rooted trees
as the reduced incidence coalgebra of a certain category of
root-preserving forest embeddings, modulo the equivalence 
relation that identifies two root-preserving forest embeddings if their 
complement crowns are isomorphic forests (see~\ref{ex:CK}).
Another prominent example fitting into D\"ur's formalism 
is the Fa\`a di Bruno bialgebra, previously obtained in \cite{JoyalMR633783}
from the category of surjections, which is however not a \M category.

Our work on Fa\`a di Bruno formulae in bialgebras of
trees~\cite{GalvezCarrillo-Kock-Tonks:1207.6404} prompted us to look for a more
general version of Leroux's theory, which would naturally realise the Fa\`a di Bruno and
Butcher--Connes--Kreimer bialgebras directly as incidence coalgebras.  A sequence of
generalisations and simplifications of the theory led to the notion of
decomposition space which is the central notion of the present work.

\subsection*{Abstraction steps: from numbers to sets, and from sets to 
$\infty$-groupoids}

The first abstraction step is to follow the objective method, pioneered in this context by
Lawvere and Menni~\cite{LawvereMenniMR2720184}, working directly with the
combinatorial objects, using linear algebra with coefficients
in $\Set$ rather than working with numbers and functions on the vector
spaces spanned by the objects.

To illustrate the objective method, observe that a vector in the free vector
space on a set $S$ is just a collection of scalars indexed by (a finite subset
of) $S$.  The objective counterpart is a family of sets indexed by $S$, i.e.~an
object in the slice category $\Set_{/S}$, and linear maps at this level are
given by spans $S \leftarrow M \to T$.  The \M inversion principle states an
equality between certain linear maps (elements in the incidence algebra).  At
the objective level, such an equality can be expressed as a bijection between
sets in the spans representing those linear functors (see the second paper in
this series~\cite{GKT:DSIAMI-2}).  In this way, algebraic identities are
revealed to be just the cardinality of bijections of sets, which carry much more
information.

In the present work we take coefficients in $\Grpd$, the $\infty$-category of
$\infty$-groupoids.  The role of vector spaces is then played by slice
$\infty$-categories $\Grpd_{/S}$.  In \cite{GKT:HLA} we have developed the
necessary `homotopy linear algebra' and the notion of homotopy cardinality,
extending many results of
Baez--Hoffnung--Walker~\cite{Baez-Hoffnung-Walker:0908.4305} who worked with
$1$-groupoids.  In order to be able to recover numerical or algebraic results by
taking cardinality, suitable finiteness conditions must be imposed, but as long
as we work at the objective level, where all results and proofs are naturally
bijective, these finiteness conditions do not play an essential role.  Outside
of this introduction we are not concerned with finiteness conditions and
cardinality in the present paper, but will return to them in the second and
third papers in this series~\cite{GKT:DSIAMI-2,GKT:MI}.

The price to pay for working at the objective level is the absence of additive
inverses: in particular, \M functions cannot exist in the usual form of an
alternating sum indexed by chains of different lengths.  However, we can prove
the following explicit equivalence of $\infty$-groupoids
(cf.~\cite{GKT:DSIAMI-2}):
\begin{align*}
\zeta * \Phieven
 &\;\;\simeq\;\; \varepsilon\;\; +\;\; \zeta * \Phiodd.
\end{align*}
We shall not here go into the definition of these $\infty$-groupoids.
The point we wish to make 
is that upon taking homotopy cardinality,
under the appropriate finiteness assumptions, and putting $\mu = 
\norm{\Phieven}-\norm{\Phiodd}$,
one
recovers the usual \M inversion formula $\zeta * \mu = \varepsilon$,
which has thus been given a `bijective' interpretation.

There are two levels of finiteness conditions needed in order
to take cardinality and arrive at algebraic (numerical) 
results~\cite{GKT:DSIAMI-2}: namely, just in
order to obtain a numerical coalgebra, for each arrow $f$ and for each $n\in\N$,
there should be only finitely many decompositions of $f$ into a chain of $n$
arrows.  Second, in order to obtain also \M inversion, the following additional
finiteness condition is needed: for each arrow $f$, there is an upper bound on the 
number of non-identity arrows in a chain of arrows composing to $f$.  The latter condition 
is important in its own right, as it is the condition for the existence of a 
length filtration (cf.~\cite[\S 6]{GKT:DSIAMI-2}, useful in many 
applications~\cite{GalvezCarrillo-Kock-Tonks:1207.6404,Kock:1411.3098,Kock:1512.03027}).

The importance of chains of arrows naturally suggests a simplicial viewpoint,
regarding a category $\CC$ as a simplicial set via its nerve $N\CC$.  Leroux's
theory can be formulated in terms of simplicial sets, and many of the arguments
then rely on certain simple pullback conditions, the first being the Segal
condition which characterises categories among simplicial sets.  Most
importantly, in our exploitation of this simplicial viewpoint the
comultiplication \eqref{Delta-f} can be written in terms of $N\CC$ as a
push-pull formula, $\Delta=(d_2,d_0)\lowershriek\circ d_1\upperstar$.

The fact that combinatorial objects typically have symmetries prompted the
upgrade from sets to groupoids, in fact a substantial conceptual 
simplification~\cite{GalvezCarrillo-Kock-Tonks:1207.6404}.
This upgrade is essentially straightforward, as long as the
notions involved are taken in a correct homotopy sense: bijections of sets are replaced
by equivalences of groupoids; the slices playing the role of vector spaces are
homotopy slices, the pullbacks and fibres involved in the functors are homotopy
pullbacks and homotopy fibres, and the sums are homotopy sums (i.e.~colimits
indexed by groupoids, just as ordinary sums are colimits indexed by sets).  
In this setting
  one may abandon also the 
strict notion of simplicial object in favour of a pseudo-functorial analogue.
For example,
the monoidal nerve of  $(\B,+,0)$, the monoidal groupoid of finite sets and bijections under disjoint union,
is actually only a pseudofunctor $\mathbf B:\simplexcategory\op\to\onegrpd$
(see \ref{monoidalgroupoids}--\ref{prop:BM}).
This level of abstraction allows us to state for example that the incidence
algebra of $\mathbf B$ is the category of species with the Cauchy 
product~\cite[\S 2.1]{GKT:ex}
(suggested as an exercise by Lawvere and Menni~\cite{LawvereMenniMR2720184}).

While it is doable to handle the $2$-category theory involved to
deal with groupoids, pseudo-functors, pseudo-natural isomorphisms, and so on,
much conceptual clarity is obtained by passing immediately to
the $\infty$-category $\Grpd$ of $\infty$-groupoids: thanks to the monumental effort of 
Joyal~\cite{Joyal:qCat+Kan}, \cite{Joyal:CRM}, Lurie~\cite{Lurie:HTT} and others,
$\infty$-groupoids can now be handled efficiently. At the
elementary level where we work, all that is needed is some basic knowledge
about (homotopy) pullbacks and (homotopy) sums, and everything looks very much like
the category of sets.  So we work throughout with certain simplicial
$\infty$-groupoids.  The appropriate notion of weak category in $\infty$-groupoids is that of Rezk
complete Segal space~\cite{Rezk:MR1804411}.  Our theory at this level says that
for any Rezk complete Segal space there is a natural incidence coalgebra
defined with coefficients in $\infty$-groupoids (this is a special case of Theorem~\ref{thm:comultcoass})
and that the objective
sign-free \M inversion principle holds \cite{GKT:DSIAMI-2}.

\subsection*{The idea of decomposition spaces}

The final abstraction step, which becomes the starting point for the paper,
is to notice that in fact neither the Segal condition
nor the Rezk condition is needed in full in order to get a (co)associative
(co)algebra and a \M inversion principle.  Coassociativity follows from (in fact
is essentially equivalent to) the {\em decomposition-space axiom} (see
\S\ref{sec:decomp} for the axiom,  and the discussion at the beginning of 
\S\ref{sec:COALG} for its derivation from
coassociativity): a decomposition space
is a simplicial $\infty$-groupoid $X:\simplexcategory\op\to\Grpd$ sending active-inert pushout
squares in $\simplexcategory$ to pullbacks. 
Whereas the Segal
condition is the expression of the ability to compose morphisms, the new
condition is about the ability to decompose, which of course in general is
easier to achieve than composability.

It is likely that all incidence (co)algebras can be realised
directly (without imposing a reduction)
as incidence (co)algebras of 
decomposition spaces.
If a reduced incidence algebra construction is known,
the decomposition space can be found by analysing the reduction step.
For example, D\"ur~\cite{Dur:1986} realises the $q$-binomial
coalgebra as the reduced incidence coalgebra of the category of
finite-dimensional vector spaces over a finite field and linear injections, by
imposing the equivalence relation identifying two linear injections if their
quotients are isomorphic.  Trying to realise the reduced incidence coalgebra
directly as a decomposition space immediately leads us to
the Waldhausen \Sdot-construction, which is a general class of examples:
we show that for any abelian category or stable $\infty$-category, 
the Waldhausen \Sdot-construction is a decomposition space (which is not 
Segal), cf.~\ref{Sdot}.
Under the appropriate finiteness conditions, the resulting incidence algebras
include the Hall algebras, as well as the derived Hall algebras first
constructed by To\"en~\cite{Toen:0501343}.

Other examples of coalgebras that can be realised as incidence coalgebras of
decomposition spaces but not of categories are Schmitt's Hopf algebra of
graphs~\cite{Schmitt:1994} and the Butcher--Connes--Kreimer Hopf algebra of
rooted trees~\cite{Connes-Kreimer:9808042}.
In a sequel paper 
\cite{GKT:restriction}, these examples are subsumed as examples of decomposition
spaces induced from restriction species and directed restriction species.

\smallskip

For our present purposes, the appropriate notion of morphism between decomposition spaces is that of
\culf functor, as these induce coalgebra homomorphisms (\ref{lem:coalg-homo}).  Many
relationships between incidence coalgebras, and in particular most of the
reductions that play a central role in the classical theory (from
Rota~\cite{Rota:Moebius} and D\"ur~\cite{Dur:1986} to
Schmitt~\cite{Schmitt:1994}), are induced from \culf functors.  The
simplicial viewpoint taken in this work reveals furthermore that many of these
\culf functors are actually instances of the notion of \emph{decalage} 
\ref{Dec},
which goes back to Lawvere~\cite{Lawvere:ordinal} and Illusie~\cite{Illusie2}.
Decalage is in fact an important ingredient in the theory to relate
decomposition spaces to Segal spaces: we observe that the decalage of a
decomposition space is a Segal space~\ref{thm:decomp-dec-segal}.

\medskip

Throughout we have strived for deriving all results from elementary principles,
such as pullbacks, factorisation systems and other universal constructions.  It
is also characteristic for our approach that we are able to reduce many
technical arguments to simplicial combinatorics.  The main notions are
formulated in terms of the active-inert factorisation system in
$\simplexcategory$ (see \ref{generic-and-free}).  To establish coassociativity
we explore also $\unDelta$ (the algebraist's Delta, including the empty ordinal)
and establish and exploit a universal property of its twisted arrow
category~(\S\ref{sec:DD}).  Sequels to this paper further vindicate this
philosophy: In \cite{GKT:MI}, in order to construct the universal decomposition
space of intervals, we study the category $\Xi$ of finite strict intervals, yet
another variation of the simplex category, related to it by an adjunction.  In
\cite{GKT:restriction}, as a general method for establishing functoriality in
inert maps, we study a certain category $\newnabla$ of convex correspondences
in $\unDelta$.  These `simplicial preliminaries' are likely to have applications
also outside the theory of decomposition spaces.

\subsection*{Related work: $2$-Segal spaces of Dyckerhoff and Kapranov}

The notion of decomposition space was arrived at independently
by Dyckerhoff and Kapranov~\cite{Dyckerhoff-Kapranov:1212.3563}:
a decomposition space is essentially the same thing as what they call
a unital $2$-Segal space.  We hasten to give them full credit
for having arrived at the notion first.  Unaware of their work,
we arrived at the same notion from a very different path, and
the theory we have developed for it is mostly orthogonal to
theirs.  

The definitions are different in appearance: the definition of
decomposition space refers to preservation of certain pullbacks,
whereas the definition of $2$-Segal space (reproduced in
\ref{DK} below) refers to triangulations of convex polygons.
The coincidence of the notions was noticed by Mathieu Anel
because two of the basic results are the same: specifically, the
characterisation in terms of decalage and Segal spaces (our
Theorem~\ref{thm:decomp-dec-segal}) and the result that the
Waldhausen \Sdot-construction of a stable $\infty$-category is a
decomposition space (our Theorem~\ref{thm:WaldhausenS}) were
obtained independently (and first)
in~\cite{Dyckerhoff-Kapranov:1212.3563}.

We were motivated by rather elementary aspects of combinatorics
and quantum field theory, and our examples are all drawn from
incidence algebras and \M inversion, whereas Dyckerhoff and Kapranov
were motivated
by representation theory, geometry, and homological algebra, and
develop a theory with a much vaster range of examples in mind:
in addition to Hall algebras and Hecke algebras they find cyclic
bar construction, mapping class groups and surface geometry (see
also \cite{Dyckerhoff-Kapranov:1306.2545} and
\cite{Dyckerhoff-Kapranov:1403.5799}), construct a Quillen model
structure and relate to topics of interest in higher category
theory such as $(\infty,2)$-categories and operads.

In the end we think our contribution is just a little corner of
a vast theory, but an important little corner, and we hope that
our viewpoints and insights will prove useful also for
the rest of the theory.

\subsection*{Related work on \M categories}

Where incidence algebras and \M inversion are concerned, our work descends from
Leroux et al.~\cite{Content-Lemay-Leroux,Leroux:1975,Leroux:IJM},
D\"ur~\cite{Dur:1986}, and Lawvere--Menni~\cite{LawvereMenniMR2720184}.

There is a different notion of \M category, due to Haigh~\cite{Haigh}.  The two
notions have been compared, and to some extent unified, by
Leinster~\cite{Leinster:1201.0413}, who calls Leroux's \M inversion {\em fine}
and Haigh's {\em coarse}, as it only depends on the underlying graph of the
category.  We should mention also the $K$-theoretic \M inversion for
quasi-finite EI categories of L\"uck and collaborators \cite{Luck:1989},
\cite{Fiore-Luck-Sauer:0908.3417}.

  The classical theory of incidence algebras of posets reached a culmination
with Schmitt's paper \cite{Schmitt:1994}.
Subsequently, Ray and Schmitt~\cite{Ray-Schmitt} introduced a
certain category of {\em cell-sets} consisting of sets equipped
with an equivalence relation (and a compatible dimension function),
whose morphisms are given by a clever multiset construction.
They showed that coalgebraic structures could be defined at this
level, prior to taking free vector spaces, and in this sense
their theory can be seen as a precursor to the fullblown
objective method of slices (indeed, their multiset morphisms are
subsumed in the natural notion of linear functor between
slices).  Although cell-sets have equivalence relations
built in, they do not account for symmetries in the same
automatic way as groupoids do, and essentially their examples
are still poset based.  It seems likely that the main structures
and constructions of~\cite{Ray-Schmitt} can be subsumed in the theory of decomposition
spaces, and we hope to get the opportunity to
take up this issue on a later occasion.

\subsection*{Outline of the present paper, section by section}

We begin in Section~\ref{sec:groupoids} with a review of some elementary notions
from the theory of $\infty$-categories, to render the paper accessible also to
readers without prior experience with these notions.
Section~\ref{sec:simplprelim} contains a few preliminaries on simplicial objects
and Segal spaces, and in Section~\ref{sec:decomp} we introduce the main notion
of this work, decomposition spaces:

\begin{deff}
  A simplicial space $X:\simplexcategory\op\to\Grpd$ is called a {\em decomposition space}
  when it takes active-inert pushouts in $\simplexcategory$ to pullbacks 
  of $\infty$-groupoids.
\end{deff}
We give some equivalent pullback characterisations, and observe that
every Segal space is a decomposition space.

In Section~\ref{sec:cULF} we turn to the relevant notion of 
morphism, that of \culf functor (meaning `conservative' and 
`ULF' = `unique lifting of factorisations'):
\begin{deff}
  A simplicial map is called {\em \culf} if it is cartesian on all active maps.
\end{deff}
After some variations, we come to decalage, which features in the following
important relationship between Segal and decomposition spaces:

\medskip

\noindent
{\bf Theorem~\ref{thm:decomp-dec-segal}.} {\em A simplicial space
$X$ is a decomposition space if and only if both
$\Dectop{X}$ and $\Decbot{X}$ are Segal spaces, and the two 
comparison maps back to $X$ are \culf.}

\medskip

In Section~\ref{sec:COALG} we introduce the incidence coalgebra 
associated to a decomposition space $X$.
It is the slice
$\infty$-category $\Grpd_{/X_1}$, with the comultiplication map given by the span
$$
X_1 \xleftarrow{\;\;\,d_1\;\;\,} X_2 \xrightarrow{(d_2,d_0)} X_1\times X_1.
$$
We explain how a naive view of coassociativity provided the motivation for the
decomposition-space axioms, but to formally establish the coassociativity result
we first require more simplicial preliminaries, introduced in 
Section~\ref{sec:DD}.  In
particular we introduce the twisted arrow category $\DD$ of the category of
finite ordinals, which is monoidal under {\em external sum}.  We show that
simplicial objects in a cartesian monoidal category can be characterised as
monoidal functors on $\DD$, and characterise decomposition spaces as those
simplicial spaces whose extension to $\DD$ preserves certain pullback squares.

In Section~\ref{sec:proofofcoass} the homotopy coassociativity of the incidence
coalgebra is established exploiting the monoidal structure on $\DD$:

\medskip

\noindent
{\bf Theorem~\ref{thm:comultcoass}.}
{\em
  For $X$ a decomposition space, the slice $\infty$-category  $\Grpd_{/X_1}$ has the structure of a strong homotopy  comonoid in the symmetric monoidal $\infty$-category $\LIN$, with the comultiplication defined by the span
  $$
X_1 \xleftarrow{\;\;\,d_1\;\;\,} X_2 \xrightarrow{(d_2,d_0)} X_1\times X_1.
$$
}

\medskip

In Section~\ref{sec:functorialities} we show that \culf functors
induce coalgebra homomorphisms.  
We also comment on a certain contravariant functoriality holding for simplicial maps that are equivalences in degree zero and are relatively Segal.

\medskip

In Section~\ref{sec:monoids-a} we introduce the notion of {\em
monoidal decomposition space}, as a monoid object in the monoidal
$\infty$-category of decomposition spaces and \culf maps.
The incidence
coalgebra of a monoidal decomposition space 
is naturally a {\em bialgebra}.

\medskip

In Section~\ref{sec:ex} we give some basic examples to provide a taste of the
breadth of applications.  Further examples are expounded in detail in
\cite{GKT:restriction} and \cite{GKT:ex}.  We begin with the example of finite
sets and injections (which leads to the binomial coalgebra), to illustrate how
decalage formalises reduction processes, and how the convolution product at the
objective level is the Cauchy product of species.  Coming to examples of
decomposition spaces which are not Segal, we take a short look at Schmitt's Hopf
algebra of graphs, and revisit the running example of the
Butcher--Connes--Kreimer Hopf algebra, comparing with the construction of
D\"ur~\cite{Dur:1986}.  Finally we consider the example of finite vector spaces,
which leads to the general case of Waldhausen's \Sdot-construction and Hall
algebras.  We establish that the $S_\bullet$ construction of an abelian category or a
stable $\infty$-category is a decomposition space.  This result was first proved
by Dyckerhoff and Kapranov and constitutes a cornerstone in their
work~\cite{Dyckerhoff-Kapranov:1212.3563}, \cite{Dyckerhoff-Kapranov:1306.2545},
\cite{Dyckerhoff-Kapranov:1403.5799}, \cite{Dyckerhoff:1505.06940}, to which we
refer for the remarkable richness of this class of examples.

\subsection*{Brief summary of the four sequels to this paper}

The present paper originally formed the
first two sections of a large manuscript~\cite{GKT:1404.3202}
which has been split into altogether six 
papers of more manageable size.  We briefly 
comment on the contents of the sequels.

The long appendix of \cite{GKT:1404.3202} has become an independent
paper \cite{GKT:HLA} developing the necessary background
on homotopy linear algebra.

In paper \cite{GKT:DSIAMI-2}, the second in
the trilogy, we introduce the notion of \emph{complete} decomposition space
in order to provide a notion of nondegeneracy
necessary for the theory of \M inversion.  In this context we can consider the
linear functors $\Phi_n$ defined by spans $X_1\leftarrow \nondeg X_n \to 1$,
where $\nondeg X_n\subset X_n$ is the subspace of nondegenerate $n$-simplices,
and prove the general \M inversion principle on the objective level:
\begin{align*}
\zeta * \Phieven
 &\;\;\simeq\;\; \varepsilon\;\; +\;\; \zeta * \Phiodd.
\end{align*}

Having established this, we analyse the finiteness conditions necessary to take cardinality and obtain numerical incidence algebras, and for the \M inversion principle to descend to these $\Q$-algebras.
We identify two conditions on complete decomposition spaces: having \emph{locally finite length} and being \emph{locally finite}. Complete decomposition spaces that satisfy both finiteness conditions are called {\em \M decomposition spaces}. 
The first finiteness condition is equivalent to the existence of a certain
{\em length filtration}, which is useful in applications.  Although many
examples coming from combinatorics do satisfy this condition, it is actually a 
rather strong condition, as witnessed by the following result:

\medskip
\noindent {\em Every decomposition space with length filtration
is the left Kan extension of a semi-simplicial space.}

\medskip

This result holds for more general simplicial spaces that we 
term {\em stiff}, and we digress to establish this.
We also consider an even weaker notion of {\em split} simplicial space,
in which all face maps preserve nondegenerate simplices.
This condition is the 
analogue of the condition for categories that identities are indecomposable,
enjoyed in particular by \M categories in the sense of Leroux.

\bigskip

In paper~\cite{GKT:MI} we come to what is perhaps the deepest theorem so far in
our work.  Lawvere showed in the 1980s that there is a Hopf algebra of \M
intervals which contains the universal \M function
\cite{LawvereMenniMR2720184}.  This Hopf algebra, obtained from the collection
of all iso-classes of \M intervals, is universal for incidence
coalgebras of \M categories $X$, by virtue of the canonical coalgebra
homomorphism from the incidence coalgebra of $X$ sending an arrow in $X$ to its
factorisation interval.  The universal Hopf algebra is not, however, the
incidence coalgebra of a \M category.

We show that it {\em is} a decomposition space.  
We construct a (large) complete decomposition space $U$ 
of all `subdivided intervals', together with a canonical \culf classifying functor   $X\to U$ for any complete decomposition space $X$.
We prove that the space of \culf maps from $X$ to $U$ is connected, and conjecture that it is contractible.

If we also impose the relevant finiteness conditions, we obtain the result that
the space of all \M intervals is a \M decomposition space.  It follows that it
admits a \M inversion formula with coefficients in finite $\infty$-groupoids or
in $\Q$, and since every \M decomposition space admits a canonical \culf functor
to it, we conclude that \M inversion in every incidence algebra is
induced from this master formula.

\bigskip

In paper~\cite{GKT:restriction} we show that Schmitt coalgebras of
restriction species \cite{Schmitt:hacs} (such as graphs, matroids, posets, 
etc.)\ naturally define
decomposition spaces.  We also introduce a new notion of {\em directed 
restriction species}:
whereas ordinary restriction species are presheaves of the category of
finite sets and injections, directed restriction species are presheaves on the 
category of finite posets and convex maps.  Examples covered by this 
notion are the Butcher--Connes--Kreimer bialgebra and the Manchon--Manin bialgebra
of directed graphs.  Both ordinary and directed restriction species are shown
to be examples of a construction of decomposition spaces from what we call
sesquicartesian fibrations, certain cocartesian fibrations over the category of 
finite ordinals that are also cartesian over convex maps.

\bigskip

In paper \cite{GKT:ex} we give examples from classical (and less classical)
combinatorics.  The first batch of examples, similar 
to the binomial posets of Doubilet--Rota--Stanley~\cite{Doubilet-Rota-Stanley},
are straightforward but serve to illustrate two key points: (1) the incidence algebra
in question is realised directly from a decomposition space, without a
reduction step, and reductions are typically given by \culf functors;
(2) at the objective level, the convolution algebra is a monoidal
structure of species (specifically: the usual Cauchy product of species, the 
shuffle product of $\mathbb L$-species, the Dirichlet product of arithmetic species,
the Joyal--Street external product of $q$-species, and the Morrison `Cauchy' 
product of $q$-species).  In each of these cases, a power series representation
results from taking cardinality.
The next class of examples includes the Fa\`a di Bruno bialgebra, the
Butcher--Connes--Kreimer bialgebra of trees, with several variations, and similar
structures on directed graphs (cf.~Manchon~\cite{Manchon:MR2921530} and
Manin~\cite{Manin:MR2562767}). 
Another important class of examples is given by Hall algebras, cf.~also
\ref{Sdot} below.
We conclude the paper by computing the \M function in a few cases, and
commenting on certain cancellations that occur in the process of taking
cardinality, substantiating that these cancellations are not possible at the
objective level.  This is related to the distinction between bijections and
natural bijections.

\medskip

\subsection*{Acknowledgments}

We are indebted first of all to Andr\'e Joyal, not only for his influence
through the theory of species and the theory of quasi-categories, essential
frameworks for our contribution, but also for his interest in our work, for
many enlightening discussions and advice that helped shape it.  We have
also learned a lot from collaboration and discussions with David Gepner.
Finally we thank Mathieu Anel, Kurusch Ebrahimi-Fard, Tom Leinster, Fei Xu,
Tobias Dyckerhoff, Julie Bergner, Louis Carlier, and Rune Haugseng for valuable
feedback, and the referee for many pertinent remarks that led to improved
exposition.

\addtocontents{toc}{\protect\setcounter{tocdepth}{2}}

\section{Preliminaries on $\infty$-groupoids and $\infty$-categories}

\label{sec:groupoids}

\begin{blanko}{Groupoids and $\infty$-groupoids.}
  Although most of our motivating examples can be naturally formulated in the
  2-category $\oneGrpd$ of $1$-groupoids, we have chosen to work in the
  $\infty$-category $\Grpd$ of $\infty$-groupoids.  This is on one hand the
  natural generality of the theory, and on the other hand a considerable
  conceptual simplification: thanks to the monumental effort of
  Joyal~\cite{Joyal:qCat+Kan}, \cite{Joyal:CRM} and Lurie~\cite{Lurie:HTT}, the
  theory of $\infty$-categories has now reached a stage where it is just as
  workable as the theory of $1$-groupoids --- if not more!  It also contains
  ordinary category theory: one regards a category $\CC$ as an $\infty$-category
  via its nerve $N \CC$.  Some details can be found in \ref{inftycats} below.  The
  philosophy is that, modulo a few homotopy {\em caveats}, one is allowed to
  think as if working in the category of sets.  A recent forceful vindication of
  this philosophy is Homotopy Type Theory~\cite{HoTT}, in which a syntax that
  resembles set theory is shown to be a powerful language for general homotopy
  types.
  
  A recurrent theme in the present work is to upgrade combinatorial
  constructions from sets to $\infty$-groupoids.  To this end the first step
  consists in understanding the construction in abstract terms, often in terms
  of pullbacks and sums, and then the second step consists in copying over the
  construction to the $\infty$-setting.  The $\infty$-category theory needed
  will be accordingly elementary, and it is our contention that it should be
  feasible to read this work without prior experience with $\infty$-groupoids or
  $\infty$-categories, simply by substituting the word `set' for the word
  `$\infty$-groupoid'.  Even at the set level, our theory contributes
  interesting insight, revealing many constructions in the classical theory to
  be governed by very general principles proven useful also in
  other areas of mathematics.
  
  The following short review of some basic aspects of $\infty$-categories should
  suffice for reading this paper and its sequels.
\end{blanko}

\begin{blanko}{From posets to Rezk categories.}
  A few remarks may be in order to relate the homotopy viewpoint with classical
  combinatorics.  A $1$-groupoid is the same
  as an ordinary groupoid, and
  a $0$-groupoid is the same
  as a set.  A $(-1)$-groupoid is the same
  as a truth value: up to equivalence there exist only two
  $(-1)$-groupoids, namely the contractible groupoid (a point) and the empty
  groupoid.  A poset is essentially the same
  as a category in which all
  the mapping spaces are $(-1)$-groupoids.  An ordinary category is a category
  in which all the mapping spaces are $0$-groupoids.  Hence the theory of
  incidence algebras of posets of Rota and collaborators can be seen as the
  $(-1)$-level of the theory.  Cartier--Foata theory and Leroux theory take
  place at the $0$-level.  We shall see that in a sense the natural setting for
  combinatorics is the $1$-level, since this level naturally takes into account
  that combinatorial structures can have symmetries.  (From this viewpoint, it
  looks as if the classical theory compensates for working one level below the
  natural one by introducing reductions.)  
  It is convenient to follow this ladder to infinity: the good
  notion of category with $\infty$-groupoids as mapping spaces is
  that of Rezk complete Segal space, also called Rezk category;
  this is the level of generality of the present work.
\end{blanko}

\begin{blanko}{$\infty$-categories.}\label{inftycats}
  By $\infty$-category we mean quasi-category~\cite{Joyal:qCat+Kan}.  These are
  simplicial sets satisfying the weak Kan condition: inner horns admit a filler.
  (An ordinary category is a simplicial set in which every inner horn admits a {\em 
  unique} filler.)
  This theory has already been developed by Joyal~\cite{Joyal:qCat+Kan,Joyal:CRM} and
  Lurie~\cite{Lurie:HTT}.
  The main point, Joyal's great insight,
  is that category theory can be generalised to
  quasi-categories, and that the results look the same, although to bootstrap
  the theory very different techniques are required.  There are other
  implementations of $\infty$-categories, such as complete Segal spaces; see 
  Bergner~\cite{Bergner:0610239} for a survey.  We
  will only use results that hold in all implementations, and for this reason we
  say $\infty$-category instead of referring explicitly to quasi-categories.
  Put another way, we shall only ever distinguish quasi-categories up to
  (categorical) equivalence, and most of the constructions rely on universal 
  properties such as pullback, which in any case only determine the objects
  up to equivalence.
  Every $1$-category is also a quasi-category via its nerve.  In particular we
  have,
  for each $n\geq 0$, the $\infty$-category $\Delta[n]$ which is the nerve of
  the linearly ordered set $\{0 \leq 1 \leq \cdots \leq n\}$.
\end{blanko}

\begin{blanko}{$\infty$-groupoids.}\label{inftygrpds}
  An $\infty$-groupoid is an $\infty$-category in
  which all morphisms are invertible.  We often say `space' instead of
  $\infty$-groupoid, as they are a combinatorial substitute for topological 
  spaces up to homotopy; for example, to each object $x$ in an $\infty$-groupoid 
  $X$,
  there are associated homotopy groups $\pi_n(X,x)$ for $n>0$.
  In terms of quasi-categories,
  $\infty$-groupoids are precisely Kan complexes, i.e.~simplicial sets in which
  every horn, not just the inner ones, admits a filler.
  
  $\infty$-groupoids play the role analogous to sets in classical category
  theory.  In particular, for any two objects $x,y$ in an $\infty$-category
  $\CC$ there is (instead of a hom set) a \emph{mapping space} $\Map_\CC(x,y)$
  which is an $\infty$-groupoid.  $\infty$-categories form a (large)
  $\infty$-category denoted $\kat{Cat}_\infty$.  $\infty$-groupoids form a
  (large) $\infty$-category denoted $\Grpd$; it can be described explicitly as
  the coherent nerve of the (simplicially enriched) category of Kan complexes.
  Given two $\infty$-categories $\DD$, $\CC$, there is a \emph{functor
  $\infty$-category} $\Fun(\DD,\CC)$.  In terms of quasi-categories, the functor
  $\infty$-category is just the internal hom of simplicial sets.   As an important
  example of a functor $\infty$-category, for a given $\infty$-category $I$ we
  have the $\infty$-category of presheaves $\Fun(I\op,\Grpd)$, and there is a
  Yoneda lemma that works as in the case of ordinary categories~\cite[Lemmas
  5.1.5.2, 5.5.2.1]{Lurie:HTT}. Since $\DD$
  and $\CC$ are objects in the $\infty$-category $\kat{Cat}_\infty$ we also have
  the $\infty$-groupoid $\Map_{\kat{Cat}_\infty}(\DD,\CC)$, which can also be
  described as the maximal sub-$\infty$-groupoid inside $\Fun(\DD,\CC)$.
\end{blanko}

\begin{blanko}{Defining $\infty$-categories and sub-$\infty$-categories.}
  While in ordinary category theory one can define a category by saying what the
  objects and the arrows are (and how they compose), this from-scratch approach
  is more difficult for $\infty$-categories, as one would have to specify the
  simplices in all dimensions and verify the filler condition (that is,
  describe the $\infty$-category as a quasi-category).  In practice,
  $\infty$-categories are constructed from existing ones by general
  constructions that automatically guarantee that the result is again an
  $\infty$-category, although the construction typically uses universal
  properties in such a way that the resulting $\infty$-category is only defined
  up to equivalence.  To specify a sub-$\infty$-category of an $\infty$-category
  $\CC$, it suffices to specify a subcategory of the homotopy category of $\CC$
  (i.e.~the category whose hom sets are $\pi_0$ of the mapping spaces of $\CC$),
  and then pull back along the components functor.  What this amounts to in
  practice is to specify the objects (closed under equivalences) and specifying
  for each pair of objects $x,y$ a full sub-$\infty$-groupoid of the
  mapping space $\Map_\CC(x,y)$, also closed under equivalences, and closed 
  under composition.
\end{blanko}

\begin{blanko}{Diagrams.}\label{blanko:diagrams}
  Since arrows in an $\infty$-category do not compose on the nose (one can talk 
  about `a' composite, not `the' composite), the $1$-categorical notion of
  commutative diagram in the naive sense is not appropriate here.  Commutative
  triangle in an $\infty$-category $\CC$
  means instead `object in the functor $\infty$-category 
  $\Fun(\Delta[2],\CC)$': the $2$-dimensional face of $\Delta[2]$ is mapped to
  a $2$-cell in $\CC$ mediating between the composite of the $01$ and $12$
  edges and the long edge $02$.  Similarly, `commutative square' means
  object in the functor $\infty$-category $\Fun(\Delta[1]\times \Delta[1], 
  \CC)$.  In general, `commutative diagram of shape $I$' means object in
  $\Fun(I,\CC)$.
\end{blanko}

\begin{blanko}{Adjoints, limits and colimits.}
  There are notions of adjoint functors, limits and colimits, which behave
  essentially in the same way as these notions in ordinary category theory, and
  are characterised by universal properties up to equivalence --- although to
  set up the theory and prove the theorems, much more technical proofs are
  required.  Most importantly, the limit over an empty diagram defines
  \emph{terminal object}, which may or may not exist.  It does exist in $\Grpd$
  where it is the singleton set, or any contractible $\infty$-groupoid, in any
  case denoted $\terminal$.
\end{blanko}

\begin{blanko}{Pullbacks and fibres.}
  Central to this work is the notion of pullback: given two morphisms of
  $\infty$-groupoids $X \to B \leftarrow Y$, there is a commutative
  square
  $$\xymatrix{
  X\times_B Y \drpullback \ar[r] \ar[d] & Y \ar[d] \\
  X \ar[r] & B
  }$$
  called the pullback, an example of a limit.
  It is defined via a universal property, as
  a terminal object in a certain auxiliary $\infty$-category consisting
  of commutative squares with sides $X \to B \leftarrow Y$.
  All formal properties of pullbacks of sets carry over to $\infty$-groupoids.
  
  For an object $b$ in an $\infty$-groupoid $B$, we denote by
  $\name b: \terminal \to B$ the morphism that picks out $b$.  Given a
  morphism of $\infty$-groupoids $p:X \to B$ and an object $b\in B$,
  the fibre of $p$ over $b$ is by definition the pullback
  $$\xymatrix{
  X_b \drpullback \ar[r] \ar[d] & X \ar[d]^p \\
  \terminal \ar[r]_{\name b} & B  .
  }$$ 
\end{blanko}

\begin{blanko}{Monomorphisms.}\label{def:mono}
  A map of $\infty$-groupoids $f:X\to Y$ is a {\em monomorphism} when its
  fibres are $(-1)$-groupoids, that is, are either empty or contractible.
  In some respects, this notion behaves like for sets: for example, if $f$
  is a monomorphism, then there is a complement $Z:=Y\shortsetminus X$ such that
  $Y \simeq X + Z$.  Hence a monomorphism is essentially an equivalence
  from $X$ onto some connected components of $Y$.  On the other hand, 
  a crucial difference between sets and $\infty$-groupoids is that diagonal
  maps of $\infty$-groupoids are not in general monomorphisms.  In fact $X \to X \times 
  X$ is a monomorphism if and only if $X$ is discrete (i.e.~equivalent to a set).
\end{blanko}

\begin{blanko}{Working in the $\infty$-category of $\infty$-groupoids, versus 
    working in the model category of simplicial sets.}
  When working with $\infty$-categories in terms of quasi-categories, one often
  works in the Joyal model structure on simplicial sets (whose fibrant objects
  are precisely the quasi-categories).  This is a very powerful technique,
  exploited masterfully by Joyal~\cite{Joyal:CRM} and Lurie~\cite{Lurie:HTT}, and
  essential to bootstrap the whole theory.  In the present work, we can
  benefit from their work, and since our constructions are generally elementary,
  we do not need to invoke model structure arguments, but can get away with
  synthetic arguments.  To illustrate the difference, consider the following 
  version of the Segal 
  condition (see~\ref{segalpqr} for details): we shall formulate it and use it by 
  simply saying {\em the natural square
  $$\xymatrix{
     X_2 \ar[r]\ar[d] & X_1 \ar[d] \\
     X_1 \ar[r] & X_0
  }$$
  is a pullback}.  This is a statement taking place in the $\infty$-category of
  $\infty$-groupoids.  A Joyal--Lurie style formulation would rather take place
  in the category of simplicial sets with the Joyal model structure and say
  something like {\em the natural map $X_2 \to X_1 \times_{X_0} X_1$ is an 
  equivalence}.  Here $X_1\times_{X_0} X_1$ refers to the actual $1$-categorical
  pullback in the category of simplicial sets, which does not coincide with
  $X_2$ on the nose, but is only naturally equivalent to it.
  \end{blanko}

  The following lemma extends a familiar result in $1$-category theory and
  is used many times in our work.
\begin{lemma}\label{pbk}
  If in a prism diagram of $\infty$-groupoids
  $$
  \vcenter{\xymatrix{
   \cdot\dto \rto &  \cdot\dto \rto &  \cdot\dto \\
  \cdot\rto & \cdot \rto & \cdot 
  }}
  $$
  the outer rectangle and the right-hand square are pullbacks,
  then the left-hand square is a pullback.
\end{lemma}
A few remarks are in order.  Note that we talk about a prism, i.e.~a 
$\Delta[1]\times\Delta[2]$-diagram:  although we have only drawn two of the 
squares of the prism, there is a third, whose horizontal sides are composites
of the two indicated arrows.  The triangles of the prism are not drawn either,
because they are the fillers that exist by the axioms of quasi-categories.
The proof follows the proof in the classical case, except that instead of saying
`given two arrows such and such, there exists a unique arrow making the diagram
commute, etc.', one has to argue with equivalences of mapping spaces (or slice
$\infty$-categories).  See for example \cite[Lemma~4.4.2.1]{Lurie:HTT}.
for the dual case of pushouts.

\begin{blanko}{Homotopy sums.}
  In ordinary category theory, a colimit indexed by a discrete category (that
  is, a set) is the same thing as a sum (coproduct).  For $\infty$-categories,
  the role of sets is played by $\infty$-groupoids.  A colimit indexed by an
  $\infty$-groupoid is called a {\em homotopy sum}.  In the case of
  $1$-groupoids, these sums are ordinary sums weighted by inverses of symmetry
  factors.  Their importance was stressed in
  \cite{GalvezCarrillo-Kock-Tonks:1207.6404}: by dealing with homotopy sums
  instead of ordinary sums, the formulae start to look very much like in the case
  of sets.  For example, given a map of $\infty$-groupoids $X \to B$, we have
  that $X$ is the homotopy sum of its fibres.
\end{blanko}

\begin{blanko}{Slice categories.}
  Maps of $\infty$-groupoids with codomain $B$ form the objects of
  a slice $\infty$-category
  $\Grpd_{/B}$, which behaves very much like a slice category in ordinary
  category theory.  For example, for the terminal object $\terminal$ we have
  $\Grpd_{/\terminal} \simeq \Grpd$. 
  Again a word of warning is due: when we refer to the 
  $\infty$-category $\Grpd_{/B}$ we only refer to an object determined up to 
  equivalence of $\infty$-categories by a certain universal property (Joyal's 
  insight of defining slice categories as adjoint to a join 
  operation~\cite{Joyal:qCat+Kan}).
  In the Joyal model structure for quasi-categories, this category is 
  represented by an explicit simplicial set.  However, there is more than
  one possibility, depending on which explicit version of the join operator
  is employed (and of course these are canonically equivalent). In the works
  of Joyal and Lurie, these different versions are distinguished, and each has
  some technical advantages.  In the present work we shall only need properties
  that hold for both, and we shall not distinguish them.
\end{blanko}

\begin{blanko}{Families.}
  A map of $\infty$-groupoids $X \to B$ can be interpreted as a family of 
  $\infty$-groupoids parametrised by $B$, namely the fibres $X_b$.  Just as 
  for sets, the same family
  can also be interpreted as a presheaf $B \to \Grpd$.  Precisely,
  for each $\infty$-groupoid $S$, we have the fundamental equivalence
  $$
  \Grpd_{/B} \isopil \Fun(B,\Grpd) ,
  $$
  which takes a family $X \to B$ to the functor sending $b \mapsto X_b$. 
  In the other direction, given a functor $F:B\to\Grpd$,
  its colimit is the total space of a family $X \to B$.
\end{blanko}

\begin{blanko}{Beck--Chevalley equivalence.}
  Pullback along a map of $\infty$-groupoids $f: J \to I$
  defines an $\infty$-functor $f\upperstar :\Grpd_{/I} \to \Grpd_{/J}$.  This functor
  is right adjoint to the functor $f\lowershriek:\Grpd_{/J} \to \Grpd_{/I}$ given by
  post-composing with $f$.  (The latter construction requires some care: as
  composition is not canonically defined, one has to choose composites.
  One can check that different choices yield equivalent
  functors.)  The following Beck--Chevalley rule (push-pull formula)
  \cite{Gepner-Haugseng-Kock}
  holds for $\infty$-groupoids: given a pullback square
  $$\xymatrix{
  \cdot \drpullback \ar[r]^f \ar[d]_p & \cdot \ar[d]^q \\
  \cdot \ar[r]_g & \cdot}$$
  there is a canonical equivalence of functors 
  \begin{equation}\label{BC}
  p\lowershriek \circ f\upperstar \simeq g\upperstar \circ q\lowershriek .
 \end{equation}
\end{blanko}

\begin{blanko}{Symmetric monoidal $\infty$-categories.}
  There is a notion of symmetric monoidal $\infty$-category, but it is
  technically more involved than the $1$-category case, since in general higher
  coherence data has to be specified beyond the $1$-categorical associator and
  MacLane pentagon condition.  This theory has been developed in detail by
  Lurie \cite[Ch.2]{Lurie:HA}, subsumed in the
  general theory of $\infty$-operads.  In the present work, a few monoidal
  structures play an important role, but since they are directly induced by
  cartesian product, we have preferred to deal with
  them in an informal (and possibly not completely rigorous) way, 
  with the same freedom as one deals with cartesian products in
  ordinary category theory.  The following case is the most important
  for our theory.  It is defined rigorously in \cite{GKT:HLA}, as a
  straightforward consequence of results of Lurie.
\end{blanko}

\begin{blanko}{The symmetric monoidal $\infty$-category $\LIN$.}
  \label{bla:LIN}
  The $\infty$-categories of the form $\Grpd_{/I}$ form the objects
  of a symmetric monoidal $\infty$-category $\LIN$, described in
  detail in \cite{GKT:HLA}: the morphisms are the linear functors,
  meaning that they preserve homotopy sums, or equivalently indeed
  all colimits.  Such functors are given by spans: the span
  $$
  I \stackrel p \leftarrow M \stackrel q \to J$$
  defines the linear functor
  $$
  q\lowershriek \circ p\upperstar : \Grpd_{/I}  \longrightarrow  \Grpd_{/J} .
  $$
  The $\infty$-category $\LIN$ plays the role of the category of
  vector spaces (although to be strict about that interpretation,
  and in particular to entertain a notion of cardinality to 
  embody the analogy,
  certain finiteness conditions should be imposed --- these play no
  essential role in the present paper).
  
  The symmetric monoidal structure on $\LIN$ is easy to describe on objects:
  $$
  \Grpd_{/I} \tensor \Grpd_{/J} = \Grpd_{/I\times J}
  $$
  just as the tensor product of vector spaces with bases indexed by sets $I$ 
  and $J$ is the vector space with basis indexed by $I\times J$.
  The neutral object is $\Grpd$.
\end{blanko}

\section{Simplicial preliminaries and Segal spaces}

\label{sec:simplprelim}

Our work relies heavily on simplicial machinery.  We briefly review the
notions needed, to establish conventions and notation. 

\begin{blanko}{The simplex category (the topologist's Delta).}
  Recall that the `simplex category' $\simplexcategory$ is the category whose objects are the nonempty
  finite ordinals
  $$
  [k] := \{ 0,1,2,\ldots, k\} ,
  $$
  and whose morphisms are the monotone maps.  These are generated by
  the coface maps $d^i : [n-1]\to [n]$, which are the monotone
  injective functions for which $i \in [n]$ is not in
  the image, and codegeneracy maps $s^i:[n+1] \to [n]$, which are
  monotone surjective functions for which $i \in [n]$ has a double preimage.
  We write $d^\bot:=d^0$ and
  $d^\top:=d^n$ for the outer coface maps.
\end{blanko}

\begin{blanko}{Active and inert maps (generic and free maps).}\label{generic-and-free}
  The category $\simplexcategory$ has an active-inert factorisation system.
  An arrow of $\simplexcategory$ is termed \emph{active} (also called {\em
  generic}), and written $g:[m]\genmap [n]$, if it preserves end-points,
  $g(0)=0$ and $g(m)=n$.  An arrow is termed \emph{inert} (also called {\em
  free}), and written $f: [m]\rat [n]$, if it is distance preserving,
  $f(i+1)=f(i)+1$ for $0\leq i\leq m-1$.  The active maps are generated by
  the codegeneracy maps and the inner coface maps, and the inert maps are
  generated by the outer coface maps.  Every arrow in $\simplexcategory$
  factors uniquely as an active map followed by an inert map, as detailed
  below.
\end{blanko}

\begin{blanko}{Background remarks.}
  The notions of generic and free maps are general notions in category theory,
  introduced by
  Weber~\cite{Weber:TAC13,Weber:TAC18}, who extracted the notions from earlier
  work of Joyal~\cite{Joyal:foncteurs-analytiques}.  A recommended entry point
  to the theory is
  Berger--Melli\`es--Weber~\cite{Berger-Mellies-Weber:1101.3064}.  The notions
  make sense for example whenever there is a cartesian monad on a presheaf
  category $\CC$: in the Kleisli category, the free maps are those from $\CC$,
  and the generic maps are those generated by the monad.  In practice, this is
  restricted to a suitable subcategory of combinatorial nature.  In the case at
  hand the monad is the free-category monad on the category of directed graphs,
  and $\simplexcategory$ arises as the restriction of the Kleisli category to
  the subcategory of non-empty linear graphs.  Other important instances of
  generic-free factorisation systems are found in the category of rooted trees
  \cite{Kock:0807} (where the monad is the free-operad monad), the category of
  Feynman graphs~\cite{Joyal-Kock:0908.2675} (where the monad is the
  free-modular-operad monad), the category of directed
  graphs~\cite{Kock:1407.3744} (where the monad is the free-properad monad), and
  Joyal's cellular category $\Theta$ \cite{Berger:Adv} (where the monad is the
  free-omega-category monad).  The more recent terminology
  `active/inert' is due to Lurie~\cite{Lurie:HA}, and is more suggestive for the
  role the two classes of maps play.
\end{blanko}

\begin{blanko}{Amalgamated ordinal sum.}\label{bla:pm}
   The {\em amalgamated ordinal sum over $[0]$} of two
   objects $[m]$ and $[n]$, denoted $[m]\intconcat [n]$,
   is given by the pushout of inert maps
   \begin{equation}\label{pm}\vcenter{\xymatrix@C+1.2pt{
   [0] \ar@{ >->}[r]^{(d^\top)^n} \ar@{ >->}[d]_{(d^\bot)^m} & [n]
   \ar@{ >->}[d]^{(d^\bot)^m}\\
   [m] \ar@{ >->}[r]_(0.4){(d^\top)^n} & [m]\intconcat[n] \makebox[0em][l]{${}=[m+n]$}\ulpullback
   }}\end{equation}
   This operation is not functorial on all maps in $\simplexcategory$, 
   but on the subcategory $\Deltaact$ of active maps it is functorial
   and defines a monoidal structure on $\Deltaact$
   (dual to ordinal sum (cf.~Lemma~\ref{lem:Delta-duality})).

   The inert maps $f:[n]\rat [m]$ are precisely the maps
   that can be written
    $$
    f:[n]\rat [a]\intconcat[n]\intconcat[b].
    $$
    Every active map with source $[a]\intconcat[n]\intconcat[b]$ splits as
    $$
    (\xymatrix@!C{{}[a]\ar@{ ->|}[r]^{g_1}&[a']})
    \;\intconcat\;
    (\xymatrix@!C{{}[n]\ar@{ ->|}[r]^{g}&[k]})
    \;\intconcat\;
    (\xymatrix@!C{{}[b]\ar@{ ->|}[r]^{g_2}&[b']})
    $$

    With these observations we can be explicit about the
    active-inert factorisation:
\end{blanko}

\begin{lemma}\label{lem:genfactsplit}
  With notation as above, the active-inert factorisation of the 
  composite of an inert map $f$
  followed by an active map $g_1\intconcat g\intconcat g_2$ is given by
    \begin{equation}\label{gffg}\vcenter{\xymatrix{
     [n] \ar@{ >->}[r]^-f\ar@{->|}[d]_g & [a]\intconcat[n]\intconcat[b] \ar@{->|}[d]^{g_1\intconcat 
     g\intconcat g_2} \\
     [k]\ar@{ >->}[r] & [a']\intconcat[k]\intconcat[b']
  }}\end{equation}
\end{lemma}
\begin{blanko}{Identity-extension squares.}
   A square \eqref{gffg} in which $g_1$ and $g_2$ are identity maps is called an  {\em identity-extension square}.
\end{blanko}
\begin{lemma}\label{genfreepushout}
  Active and inert maps in $\simplexcategory$ admit pushouts along each other, and the 
  resulting maps are again active and inert. In fact, active-inert pushouts are precisely the identity extension squares.
\begin{equation*}
\vcenter{\xymatrix{
     [n] \ar@{ >->}[r]\ar@{->|}[d]_{g} & [a]\intconcat[n]\intconcat[b] \ar@{->|}[d]^{\id\intconcat g\intconcat\id} \\
     [k]\ar@{ >->}[r] & [a]\intconcat[k]\intconcat[b]
}}\end{equation*}
\end{lemma}

These pushouts are fundamental to this work.  We will define
decomposition spaces to be simplicial spaces $X:\simplexcategory\op\to\Grpd$
that send these pushouts to pullbacks.

The previous lemma has the following easy corollary.
\begin{cor}\label{d1s0Delta}
  Every codegeneracy map is a pushout (along an inert map) of $s^0:[1]\to[0]$,
  and every active coface maps is a pushout (along an inert map) of 
  $d^1:[1]\to[2]$.
\end{cor}

\begin{blanko}{Simplicial spaces and Segal spaces.}\label{simpl-Grpd}
  Our main object of study will be simplicial $\infty$-groupoids subject to
  various exactness conditions, all formulated in terms of pullbacks.
  More precisely we work in the functor $\infty$-category
  $$
  \Fun(\simplexcategory\op,\Grpd) ,
  $$
  whose objects are functors  $X:\simplexcategory\op\to\Grpd$, from the $\infty$-category $\simplexcategory\op$ to the
  $\infty$-category $\Grpd$.  We prefer to call these objects {\em simplicial spaces} rather than simplicial $\infty$-groupoids.
 As explained in \ref{blanko:diagrams} the simplicial
  identities for $X$
  are not strictly commutative squares but $\Delta[1]\times \Delta[1]$-diagrams in $\Grpd$, hence come
  equipped with a homotopy between the two ways around in the square.  But this
  is precisely the setting for pullbacks.

Consider a simplicial space $X: \simplexcategory\op \to \Grpd$.
We recall the {\em Segal maps}
$$
(\partial_{0,1},\dots,\partial_{r-1,r}):X_r \longrightarrow X_1 \times_{X_0} 
\cdots \times_{X_0} X_1, \qquad r\geq 0,
$$
where $\partial_{k-1,k}:X_r\to X_1$ is induced by the map 
$[1]\rat[r]$ 
sending 0,1  to $k-1,k$.

A {\em Segal space} is a simplicial space satisfying the Segal 
condition,
namely that the Segal maps are equivalences.
 (This is automatic for $r=0,1$ as the Segal map is just the 
 identity map $X_r\to X_r$, by convention).
\end{blanko}

\begin{lemma}\label{segalpqr}
  The following conditions are equivalent, for any simplicial space $X$:
\begin{enumerate}
\item $X$ satisfies the Segal condition,
$$X_r \stackrel\simeq\longrightarrow X_1 \times_{X_0} \cdots \times_{X_0} X_1 \qquad \text{for all }r\geq 0.$$
\item The following square is a pullback for all $p,q\geq r\geq 0$
$$
\vcenter{\xymatrix{
   X_{p-r+q}\dto_{{{d_{p+1}}^{q-r}}}
\drpullback 
\rto^-{{d_{0}}^{p-r}}
 &  X_q\dto^{{{d_{r+1}}^{q-r}}} \\
   X_{p}\rto_-{{d_{0}}^{p-r}}  &  X_{r}\,.
  }}
$$
\item The following square is a pullback for all $n>0$
$$
\vcenter{\xymatrix{
   X_{n+1}\dto_{d_\top}\drpullback 
\rto^-{d_\bot}
 &  X_n\dto^{d_\top} \\
   X_{n}\rto_-{d_\bot}  &  X_{n-1}\,.
  }}
$$
\item The following square is a pullback for all $p,q\geq 0$
$$
\vcenter{\xymatrix{
   X_{p+q}\dto_{{{d_{p+1}}^{q}}}
\drpullback 
\rto^-{{d_{0}}^{p}}
 &  X_q\dto^{{{d_{1}}^{q}}} \\
   X_{p}\rto_-{{d_{0}}^{p}}  &  X_{0}\,.
  }}
$$
\end{enumerate}
\end{lemma}
\begin{proof}
It is straightforward to show that the Segal condition implies (2). 
Now (3) and (4) are special cases of (2). Also (3) implies (2):
the pullback in (2) is a composite of pullbacks of the type given in (3). 
Finally one shows inductively that (4) implies the Segal condition (1).
\end{proof}  

A simplicial map $F:Y\to X$ is {\em cartesian} on
an arrow $[n]\to[k]$ in $\simplexcategory$ if the naturality square for $F$ with respect to this arrow is a pullback.
\begin{lemma}\label{cart/Segal=Segal}
    If a simplicial map $F: Y \to X$ is 
    cartesian on outer coface maps, and if $X$ 
    is a Segal space, then $Y$ is a Segal space too.
\end{lemma}

\begin{blanko}{Rezk completeness.}
  Let $J$ denote the (ordinary) nerve of the 
  groupoid generated by one isomorphism $0 \to 1$.  
  A Segal space $X$ is {\em Rezk complete} when 
  the natural map
  $$
  \Map(\terminal, X) \to \Map(J,X)
  $$
  (obtained by precomposing with $J \to \terminal$)
  is an equivalence of $\infty$-groupoids.
  It means that the space of identity arrows is
  equivalent to the space of equivalences.  
  (See \cite[Thm.6.2]{Rezk:MR1804411}, \cite{Bergner:0610239} 
  and \cite{Joyal-Tierney:0607820}.) 
  A Rezk complete Segal space is also called a {\em Rezk category}.
\end{blanko}

\begin{blanko}{Ordinary nerve.}
    Let $\CC$ be a small $1$-category.  The {\em nerve} of $\CC$ is the simplicial set
  \begin{eqnarray*}
 N\CC:   \simplexcategory\op & \longrightarrow & \Set  \\
    {}[n]  & \longmapsto & \Fun([n],\CC) ,
  \end{eqnarray*}
  where $\Fun([n],\CC)$ is the {\em set} of strings of $n$ composable arrows.
  Sub\-examples of this are given by any poset or any monoid.  The simplicial
  sets that arise like this are precisely those satisfying the Segal condition
  (which is strict in this context).  If each set is regarded as a discrete
  $\infty$-groupoid, $N\CC$ is thus a Segal space.  In general it is not
  Rezk complete, since some object may have a nontrivial automorphism.
As an example, if $\CC$ is a
  one-object groupoid (i.e.~a group), then inside $(N\CC)_1$ the space of
  equivalences is the whole set $(N\CC)_1$, but the degeneracy map $s_0 :
  (N\CC)_0 \to (N\CC)_1$ is not an equivalence (unless the group is trivial).
\end{blanko}

\begin{blanko}{The fat nerve of an essentially small $1$-category.}\label{fatnerve}
  In most cases it is more interesting to consider the
  {\em fat nerve}. Given a $1$-category $\CC$, the fat nerve of $\CC$ is the simplicial $1$-groupoid 
  \begin{eqnarray*}
    \fatnerve\CC :\simplexcategory\op & \longrightarrow & \oneGrpd \\
    {} [n] & \longmapsto & \Map([n],\CC),
  \end{eqnarray*}
  where $\Map([n],\CC)$ is the mapping space, defined as the
  maximal subgroupoid of the functor $1$-category $\Fun([n],\CC)$.
  In other words, $(\fatnerve\CC)_n$ is the $1$-groupoid
  whose objects are strings of $n$ composable arrows in $\CC$ and
  whose morphisms are isomorphisms of such strings:
   $$
   \xymatrix{
   \cdot \ar[r] \ar[d]^*-[@]=0+!L{\scriptstyle \sim} & \cdot \ar[d]^*-[@]=0+!L{\scriptstyle \sim}
   \ar[r]  & \cdot \ar[r] \ar[d]^*-[@]=0+!L{\scriptstyle \sim} &\cdots\ar[r]&
    \cdot \ar[d]^*-[@]=0+!L{\scriptstyle \sim} \\
   \cdot \ar[r] & \cdot\ar[r] & \cdot\ar[r] &\cdots\ar[r]& \cdot
   }$$
   It is straightforward to check the Segal condition, remembering that the
   pullbacks involved are homotopy pullbacks.  For instance, the pullback
   $X_1\times_{X_0} X_1$ has as objects strings of `weakly composable'
   arrows, in the sense that the target of the first arrow is isomorphic
   to the source of the second, and a comparison isomorphism is specified.  The
   Segal map $X_2 \to X_1 \times_{X_0} X_1$ is the inclusion of the subgroupoid
   consisting of strictly composable pairs.  But any
   weakly composable pair is isomorphic to a strictly composable pair, and
   the comparison isomorphism is unique, hence the inclusion $X_2 \into
   X_1\times_{X_0} X_1$ is an equivalence.  Furthermore, the fat nerve is Rezk 
   complete.  Indeed, it is easy to see that inside $X_1$, the equivalences
   are the invertible arrows of $\CC$. But any invertible arrow is
   equivalent to an identity arrow.

  Note that if $\CC$ is a category with no non-trivial
  isomorphisms (e.g.~any \M category in the sense of Leroux)
  then the fat nerve coincides with the ordinary nerve, and if
  $\CC$ is just equivalent to such a category
  then the fat nerve is level-wise equivalent to the ordinary nerve of any
  skeleton of $\CC$.
\end{blanko}

\begin{blanko}{Joyal--Tierney $t\uppershriek$ --- the fat nerve of an 
  $\infty$-category.}\label{fatinfty}
  The fat nerve construction is just a special case of the general
  construction $t\uppershriek$ of Joyal and Tierney~\cite{Joyal-Tierney:0607820},
  which is a functor from quasi-categories to complete Segal spaces, meaning 
  specifically certain simplicial objects in the category of Kan complexes:
  given a quasi-category $\CC$, the complete Segal space $t\uppershriek \CC$
  is given by
  \begin{eqnarray*}
    \simplexcategory\op & \longrightarrow & \kat{Kan}  \\
    {}[n] & \longmapsto & \big[ [k] \mapsto \kat{sSet}(\Delta[n] \times \Delta'[k], 
    \CC) \big]
  \end{eqnarray*}
  where $\Delta'[k]$ denotes (the ordinary nerve of) the groupoid freely generated by a string of $k$
  invertible arrows.  They show that $t\uppershriek$ constitutes in fact a (right)
  Quillen equivalence between the simplicial sets with the Joyal model
  structure, and bisimplicial sets with the Rezk model structure. 
  
  Taking a more invariant viewpoint, talking about $\infty$-groupoids
  abstractly, the Joyal--Tierney $t\uppershriek$ functor associates to an
  $\infty$-category $\CC$ the Rezk complete Segal space
  \begin{eqnarray*}
    \fatnerve\CC:\simplexcategory\op & \longrightarrow & \Grpd  \\
    {}[n] & \longmapsto & \Map(\Delta[n], \CC) .
  \end{eqnarray*}
  If $\CC$ is a $1$-category regarded as an $\infty$-category 
  (via its ordinary nerve) this agrees with the fat nerve~\ref{fatnerve} 
  regarded as a simplicial $\infty$-groupoid.
\end{blanko}

\begin{blanko}{Fat nerve of bicategories with only invertible $2$-cells.}
  From a bicategory $\CC$ with only invertible $2$-cells one can get a
  simplicial bigroupoid by a construction analogous to the fat nerve.  (In fact,
  this can be viewed as the $t\uppershriek$ construction applied to the Duskin
  nerve of $\CC$.)
  The {\em fat nerve} of a bicategory $\CC$ is the
  complete Segal bigroupoid
\begin{eqnarray*}
\fatnerve\CC:\simplexcategory\op & \longrightarrow & \kat{2Grpd} \\
{}[n] & \longmapsto & \PsFun([n],\CC) ,
\end{eqnarray*}
the bigroupoid of normalised pseudofunctors.
\end{blanko}

\begin{blanko}{Monoidal groupoids.}\label{monoidalgroupoids}
  Important examples of the previous situation come from monoidal groupoids
  $(\MM,\tensor,I)$. Consider $\MM$ as a one-object bicategory $B\MM$ with composition $\tensor$. This is often termed the classifying space of $\MM$. Applying the fat nerve yields a Segal bigroupoid $\fatnerve (B\MM)$, as above, whose zeroth space is the
  classifying space of the full subgroupoid $\MM^{\eq}$ spanned by the
  tensor-invertible objects.
  
  The fat nerve construction can be simplified considerably in the case that $\MM^{\eq}$
  is contractible.  This happens precisely when every tensor-invertible object is
  isomorphic to the unit object $I$ and $I$ admits no
  non-trivial automorphisms.
\end{blanko}

\begin{prop}\label{prop:BM}
  If $(\MM,\tensor,I)$ is a monoidal groupoid such that $\MM^\eq$ is 
  contractible, then the Segal bigroupoid $\fatnerve B\MM$ is
  equivalent to the monoidal nerve: the simplicial $1$-groupoid
  \begin{eqnarray}\label{constr:BM}
    \simplexcategory\op & \longrightarrow & \oneGrpd  \\
    {}[n] & \longmapsto & \MM\times\MM\times\dots\times\MM\nonumber
  \end{eqnarray}
  where the outer face maps project away an outer factor, the inner face maps
  tensor together two adjacent factors, and the degeneracy maps insert a neutral 
  object. This weakly simplicial 1-groupoid is strictly simplicial if and only if the monoidal structure of $\MM$ is strict.
\end{prop} 
\noindent We have omitted the proof, to avoid going into $2$-category theory.

  Examples of monoidal groupoids satisfying the conditions of the proposition
  are the monoidal groupoid $(\B, +, 0)$ of finite sets and bijections, or the
  monoidal groupoid of vector spaces and linear isomorphisms under direct sum.
  In contrast, the monoidal groupoid of vector spaces and linear isomorphisms
  under tensor product is not of this kind, as the unit object $\ground$ has
  many automorphisms.  In this case the monoidal nerve \eqref{constr:BM} gives a
  Segal $1$-groupoid that is not Rezk complete.

\section{Decomposition spaces}
\label{sec:decomp}

Recall from Lemma~\ref{genfreepushout} that active and inert maps in $\simplexcategory$ admit pushouts along each other.
\begin{blanko}{Definition.}\label{def:decomp}
  A \emph{decomposition space} is a simplicial space 
  $$
  X:\simplexcategory\op\to\Grpd
  $$ 
such that the image of any pushout diagram in $\simplexcategory$ of an active map $g$ along 
an inert map $f$ is a pullback of $\infty$-groupoids,
$$ X\!\left(\!\! 
  \vcenter{\xymatrix{{}
   [p] \drpullback{} &  [m]\ar@{->|}[l]_{g'}  \\
   [q]\ar@{ >->}[u]^-{f'}   & \ar@{->|}[l]^{g} [n] \ar@{ >->}[u]_-{f}
  }}
\right)
\qquad=\qquad  \vcenter{\xymatrix{
   X_{p}\ar[d]_{{f'}^*} \ar[r]^-{{g'}^*} \drpullback&  X_m\ar[d]^{{f}^*}  \\
   X_q\ar[r]_-{{g}^*}   &  X_n .
  }}
$$
\end{blanko}

\begin{blanko}{Remark.}\label{DK}
   The notion of decomposition space can be seen as an 
   abstraction of coalgebra, cf.~\S\ref{sec:COALG} below:
   it is precisely the condition required to obtain a counital
   coassociative comultiplication on $\Grpd_{/X_1}$.

  The notion 
  is equivalent to the notion of
  unital (combinatorial) $2$-Segal space introduced by Dyckerhoff and 
  Kapranov~\cite{Dyckerhoff-Kapranov:1212.3563} (their Definition~2.3.1, 
  Definition~2.5.2, Definition 5.2.2, Remark~5.2.4).  Briefly, their
  definition goes as follows.  For any triangulation $T$ of a convex polygon
  with $n$ vertices, there is induced a simplicial subset $\Delta^T \subset 
  \Delta[n]$.  A simplicial space $X$ is called $2$-Segal if, for every
  triangulation $T$ of every convex $n$-gon, the induced map $\Map(\Delta[n],X) \to 
  \Map(\Delta^T,X)$ is a weak homotopy equivalence.
  Unitality is defined separately in terms of pullback conditions involving 
  degeneracy maps, similar to our \eqref{unital-cond} below.  The equivalence between
  decomposition spaces and unital $2$-Segal spaces follows from 
  Proposition~2.3.2 of \cite{Dyckerhoff-Kapranov:1212.3563}
  which gives a pullback criterion for the $2$-Segal condition.
\end{blanko}

\begin{blanko}{Running example: the decomposition space of rooted trees.}\label{ex:trees-decomp}
  We give an example, briefly previewed in \ref{ex:preintro}, of a decomposition space which is not a Segal space,
  to illustrate the combinatorial meaning of the pullback condition: it is
  about structures that can be decomposed but not always composed.
  This example corresponds 
  to the Butcher--Connes--Kreimer Hopf algebra of trees, as will shall see
  when we return to it in \ref{ex:trees-coalg}. 
  
  We define a simplicial groupoid $\dstrees$: take $\dstrees_1$ to be the
  groupoid of forests and, more generally, let $\dstrees_k$ be the groupoid of forests
  with $k-1$ compatible admissible cuts, partitioning the forest into $k$ layers
  (which may be empty), numbered from leaves to the root.
  Thus $\dstrees_0$ is the trivial groupoid with one object: the empty forest.
  
  These groupoids form a simplicial object: the
  outer face maps delete the bottom or the top layer, and the inner face
  maps join adjacent layers.  The degeneracy maps insert an empty layer 
  (i.e.~duplicate an admissible cut).
  The simplicial identities are obviously verified, and one can easily check that $\dstrees$ is in fact a decomposition space. Having the pullback
  $$\xymatrix{
    \dstrees_2 \ar[d]_{d_2} & \dstrees_3 \dlpullback \ar[l]_{d_1} \ar[d]^{d_3} \\
    \dstrees_1 & \dstrees_2 \ar[l]^{d_1}
  }$$
  means any tree with two compatible admissible cuts ($\in \dstrees_3$)
  is uniquely determined by a pair of elements in $\dstrees_2$ with common 
  image in $\dstrees_1$ (under the indicated face maps). 
  For example, the following picture represents elements corresponding to each other
  in the four groupoids.

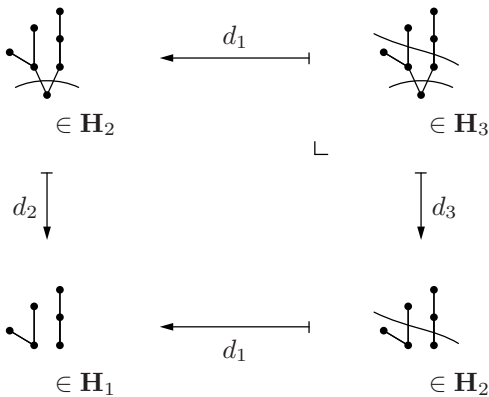
\begin{center}
    \vspace*{12pt}
    \begin{texdraw}
      \setunitscale 0.7
      \footnotesize
      
      \move (0 0)
      \bsegment
	\htext (20 -5){$\in \dstrees_1$}
	\move (-7 15) \twoskewtr
	\move (7 15) \twolintr
      \esegment

      \move (200 0)
      \bsegment
	\htext (20 -5){$\in \dstrees_2$}	
		\move (-7 15) \twoskewtr
	\move (7 15) \twolintr
	\move (0 15)
	\uppercut
      \esegment

      \move (0 150)
      \bsegment
	\htext (20 -15){$\in \dstrees_2$}
	\move (0 0) \twoVtr
	\lowercut
	\move (-7 15) \twoskewtr
	\move (7 15) \twolintr
      \esegment

      \move (200 150)
      \bsegment
	\htext (20 -15){$\in \dstrees_3$}
	\move (0 0) \twoVtr
	\lowercut
	\move (-7 15) \twoskewtr
	\move (7 15) \twolintr
	\move (0 15)
	\uppercut
      \esegment

      \move (143 118)
      \bsegment
	\move (0 0) \lvec (0 7)
	\move (0 0) \lvec (7 0)
      \esegment
     
      \move (100 170)
      \bsegment
      \move (40 -3) \rlvec (0 6)
      \move (40 0) \ravec (-80 0) 
      \htext (0 12) {$d_1$}
      \esegment

      \move (100 25)
      \bsegment
      \move (40 -3) \rlvec (0 6)
      \move (40 0) \ravec (-80 0) 
      \htext (0 -12) {$d_1$}
      \esegment

      \move (0 75)
      \bsegment
      \move (-3 33) \rlvec (6 0)
      \move (0 33) \avec (0 -3) 
      \htext (-12 15) {$d_2$}
      \esegment

      \move (200 75)
      \bsegment
      \move (-3 33) \rlvec (6 0)
      \move (0 33) \avec (0 -3) 
      \htext (12 15) {$d_3$}
      \esegment

    \end{texdraw}
    \vspace*{12pt}
  \end{center}
  The horizontal maps join layers one and two (i.e.~forget the first 
  admissible cut).  The vertical
  maps discard the last layer.  Clearly the diagram commutes.  To reconstruct the
  tree with two admissible cuts (upper right-hand corner), most of the
  information is already available in the upper left-hand corner, namely the
  underlying tree and one of the cuts.  But the remaining cut
  is precisely available in the lower right-hand
  corner, and their common image in $\dstrees_1$ says precisely how this missing piece
  of information is to be implanted.

  Note that $\dstrees$ is not a
  Segal space: in the diagram above there is a forest with a cut where the two layers do \emph{not} determine the forest.
  Thus the square
    $$\xymatrix{
      \dstrees_2 
      \ar[r]^{d_0}\ar[d]_{d_2} & \dstrees_1 \ar[d]^{d_1} \\
     \dstrees_1 \ar[r]_{d_0} & \dstrees_0\makebox[0em][l]{${}=1$}
   }$$
   is not a pullback square as required by the Segal condition~\ref{segalpqr}~(4) (with $p=q=1$).
\end{blanko}

\begin{blanko}{Alternative formulations of the pullback condition.}
  To verify the conditions of the definition, it will in fact be sufficient to
  check a smaller collection of squares.  On the other hand, the definition will
  imply that many other squares of interest are pullbacks too.  The formulation
  in terms of active and inert maps is preferred both for practical reasons and
  for its conceptual simplicity compared to the smaller or larger collections of
  squares.

  Recall from Lemma \ref{genfreepushout} that the active-inert pushouts used in
  the definition are just the identity extension squares,
  $$\xymatrix{
       [n] \ar@{ >->}[d]\ar@{->|}[rr]^-g &&[k] \ar@{ >->}[d] \\
  [a]\intconcat[n]\intconcat[b]     \ar@{->|}[rr]_-{\id\intconcat g\intconcat\id} && [a]\intconcat[k]\intconcat[b]\,.}$$
  Such a square can be written as a vertical composite of squares in which either $a=1$ and $b=0$, or vice-versa. In turn, since
  the active map $g$ 
  is a composite of inner coface maps $d^i:[m-1]\to[m]$ ($0<i<m$) and codegeneracy maps $s^j:[m+1]\to[m]$, these squares are horizontal composites of pushouts of a single active $d^i$ or $s^j$ along $d^\bot$ or $d^\top$. 
  Thus, to check that $X$ is a decomposition space, it is sufficient to check the following special cases
  are pullbacks, for $0<i<n$ and  $0\leq j\leq n$:
  $$  
  \xymatrix{
   X_{1+n}\ar[d]_{d_\bot}\drpullback 
\ar[r]^-{d_{1+i}}
 &  X_n\ar[d]^{d_\bot} \\
   X_{n}\ar[r]_-{d_i}  &  X_{n-1},
  }
\qquad
\xymatrix{
   X_{n+1}\ar[d]_{d_\top}\drpullback 
\ar[r]^-{d_{i}}
 &  X_n\ar[d]^{d_\top} \\
   X_{n}\ar[r]_-{d_i}  &  X_{n-1},
  }
$$
\begin{equation}\label{unital-cond}
\vcenter{\xymatrix{
   X_{1+n}\ar[r]^{s_{1+j}}\drpullback 
\ar[d]_-{d_\bot}
 & X_{1+n+1}\ar[d]^{d_\bot} \\
  X_n  \ar[r]_-{s_j}  &  X_{n+1},
}}\qquad
\vcenter{\xymatrix{
   X_{n+1}\ar[d]_{d_\top}\drpullback 
\ar[r]^-{s_{j}}
 &  X_{n+1+1} \ar[d]^{d_\top} \\
X_n  \ar[r]_-{s_j}  &  X_{n+1}.
  }}
\end{equation}

The following proposition shows we can be more economic: instead of checking all $0<i<n$ it is enough 
to
check all $n\geq 2$ and {\em some} $0<i<n$, and instead of checking all
$0\leq j \leq n$ it is enough to check the case $j=n=0$.
\end{blanko}

\begin{prop}\label{onlyfourdiags}
A simplicial space $X$ is a decomposition space if and only if 
the following diagrams are pullbacks
$$  
\xymatrix{
   X_1\ar[r]^{s_1}\drpullback 
\ar[d]_-{d_\bot}
 & X_2\ar[d]^{d_\bot} \\
  X_0  \ar[r]_-{s_0}  &  X_1,
}\qquad
\xymatrix{
   X_1\ar[d]_{d_\top}\drpullback 
\ar[r]^-{s_0}
 &  X_2 \ar[d]^{d_\top} \\
X_0  \ar[r]_-{s_0}  &  X_1,
  }
$$
and the following diagrams are pullbacks for some choice of $i=i_n$, 
$0<i<n$, for each $n\geq 2$:
$$  
\xymatrix{
   X_{1+n}\ar[d]_{d_\bot}\drpullback 
\ar[r]^-{d_{1+i}}
 &  X_n\ar[d]^{d_\bot} \\
   X_{n}\ar[r]_-{d_i}  &  X_{n-1},
  }
\qquad
\xymatrix{
   X_{n+1}\ar[d]_{d_\top}\drpullback 
\ar[r]^-{d_{i}}
 &  X_n\ar[d]^{d_\top} \\
   X_{n}\ar[r]_-{d_i}  &  X_{n-1}.
  }
$$
\end{prop}

\begin{proof}
  To see the non-necessity of the other degeneracy cases, observe
  that for $n>0$, every degeneracy map $s_j: X_n \to X_{n+1}$
  is the section of an {\em inner} face map
  $d_i$ (where $i=j$ or $i=j+1$).  Now in the diagram
  $$
  \xymatrix{
   X_{1+n}\ar[r]^{s_{1+j}} \ar[d]_-{d_\bot} & X_{1+n+1}\ar[d]^{d_\bot} \ar[r]^{d_{1+i}} & 
   X_{1+n} \ar[d]^{d_\bot}\\
  X_n  \ar[r]_-{s_j}  &  X_{n+1} \ar[r]_{d_i} & X_n,
  }$$
  the horizontal composites are identities, so the outer rectangle is a
  pullback, and the right-hand square is a pullback since it is one of cases
  outer face with inner face.  Hence the left-hand square, by Lemma~\ref{pbk}, is a pullback too.
  The case $s_0: X_0 \to X_1$ is the only degeneracy map that is not the section
  of an inner face map, so we cannot eliminate the two cases involving this map.
  The non-necessity of the other inner-face-map cases is the content of the
  following lemma.
\end{proof}

\begin{lemma}\label{lem:fewerdiagrams}
The following are equivalent for a simplicial space $X$. 
\begin{enumerate}
  \item  For each $n\geq2$, the following diagram is a pullback for all $0<i<n$:
$$  
\vcenter{\xymatrix{
   X_{1+n}\ar[d]_{d_\bot}\drpullback 
\ar[r]^-{d_{1+i}}
 &  X_n\ar[d]^{d_\bot} \\
   X_{n}\ar[r]_-{d_i}  &  X_{n-1}
  }}
\qquad\left(\text{resp. }\vcenter{
\xymatrix{
   X_{n+1}\ar[d]_{d_\top}\drpullback 
\ar[r]^-{d_{i}}
 &  X_n\ar[d]^{d_\top} \\
   X_{n}\ar[r]_-{d_i}  &  X_{n-1}
  }}\right){.}
$$

 \item 
For each $n\geq2$, the above diagram is a pullback for some $0<i<n$.
 \item 
For each $n\geq2$, the following diagram is a pullback:
$$
\vcenter{\xymatrix{
   X_{1+n}\ar[d]_{d_\bot}\drpullback 
\ar[r]^-{{d_2}^{n-1}}
 &  X_2\ar[d]^{d_\bot} \\
   X_n\ar[r]_-{{d_1}^{n-1}}  &  X_{1}
  }}
\qquad\left(\text{resp. }\vcenter{\xymatrix{
   X_{n+1}\ar[d]_{d_\top}\drpullback 
\ar[r]^-{{d_1}^{n-1}}
 &  X_2\ar[d]^{d_\top} \\
   X_n\ar[r]_-{{d_1}^{n-1}}  &  X_{1}
  }}\right){.}
$$
\end{enumerate}
\end{lemma}

\begin{proof} The hypothesised pullback in (2) is a special case of that in (1),
  and that in (3) is a horizontal composite of those in (2), since there is a unique active map $[1]\to[n]$ in $\simplexcategory$ for each $n$. 
The implication (3) $\Rightarrow$ (1) follows by Lemma~\ref{pbk} and the commutativity for 
$0<i<n$ of the diagram
$$
\vcenter{\xymatrix{
   X_{1+n} \drpullback \ar[r]^{d_{1+i}} \ar[d]_{d_\bot} & 
   X_n \drpullback \ar[r]^{{d_2}^{n-1}} \ar[d]_{d_\bot} &
   X_2 \ar[d]^{d_\bot} \\
   X_n \ar[r]_{d_i} &
   X_{n-1} \ar[r]_{{d_1}^{n-1}} &
   X_1\,.
   }}
   $$
Similarly for the `resp.' case.
\end{proof}

\begin{prop}\label{prop:segal=>decomp1}
  Any Segal space is a decomposition space.
\end{prop}
\begin{proof}
  Let $X$ be Segal space. In the diagram ($n\geq 2$)
  $$
\vcenter{\xymatrix{
   X_{n+1}\ar[d]_{d_\bot} 
\ar[r]^-{d_n}
 &  X_n\ar[d]_{d_\bot}\drpullback 
\ar[r]^-{d_\top}
 &  X_{n-1}\ar[d]^{d_\bot} \\
   X_{n}\ar[r]_-{d_{n-1}}  &  X_{n-1} \ar[r]_-{d_\top}  &  X_{n-2} ,
  }}
$$
since the horizontal composites are equal to $d_\top\circ d_\top$, both the outer
rectangle and the right-hand square are pullbacks by the Segal condition 
(\ref{segalpqr}~(3)).
Hence the left-hand square is a pullback.  This establishes the third
pullback condition in Proposition~\ref{onlyfourdiags}. 
In the diagram
  $$
\vcenter{\xymatrix{
   X_{1}\ar[d]_{d_\bot} 
\ar[r]^-{s_1}
 &  X_2\ar[d]_{d_\bot}\drpullback 
\ar[r]^-{d_\top}
 &  X_1\ar[d]^{d_\bot} \\
   X_0\ar[r]_-{s_0}  &  X_1 \ar[r]_-{d_\top}  &  X_0 ,
  }}
$$
since the horizontal composites are identities, the outer rectangle is a 
pullback, and the right-hand square is a pullback by the Segal condition.
Hence the left-hand square is a pullback, establishing the
first of the
pullback conditions in Proposition~\ref{onlyfourdiags}.
The remaining two conditions of Proposition~\ref{onlyfourdiags}, those
involving $d_\top$ instead of $d_\bot$, are obtained
similarly by interchanging the roles of $\bot$ and $\top$.
\end{proof}

\begin{blanko}{Remark.}
  This result was also obtained by 
  Dyckerhoff and Kapranov~\cite{Dyckerhoff-Kapranov:1212.3563}
  (Propositions~2.3.3, 2.5.3, and 5.2.6).
\end{blanko}

Corollary~\ref{d1s0Delta} implies the following important property of decomposition spaces.

\begin{lemma}\label{lem:s0d1}
    In a decomposition space $X$, every active face map is a pullback of
    $d_1: X_2 \to X_1$, and every degeneracy map is a pullback of $s_0 :X_0 \to X_1$.
\end{lemma}
Thus, even though the spaces in degree $\geq 2$
are not fibre products of $X_1$ as in a Segal space,
the higher active face maps and degeneracies 
are determined by `unit' and `composition',
$$
\xymatrix{
X_0 \ar[r]^{s_0} & X_1 & \ar[l]_{d_1} X_2 .
}$$

In $\simplexcategory\op$ there are more pullbacks than those between active and inert.
Diagram \eqref{pm} in \ref{bla:pm} is a pullback in
$\simplexcategory\op$ that is not preserved by all decomposition spaces, though it is
preserved by all Segal spaces.  
On the other hand, certain other pullbacks in
$\simplexcategory\op$ are preserved by general decomposition spaces.  We call them
colloquially `bonus pullbacks':

\begin{lemma}\label{bonus-pullbacks}
  Let $X$ be a decomposition space.
  For all $n\geq 3$ and all $0<i<j<n$, the following squares of active face and 
  degeneracy maps are pullbacks.
  $$\vcenter{\xymatrix{
     X_{n+1} \drpullback \ar[r]^-{d_i}\ar[d]_{d_{j+1}} & X_{n} \ar[d]^{d_j} \\
     X_{n} \ar[r]_-{d_i} & X_{n-1} 
  }} \ \
  \qquad
  \qquad
  \vcenter{\xymatrix{
     X_{n-3} \drpullback \ar[r]^-{s_{i-1}}\ar[d]_{s_{j-2}} & X_{n-2} \ar[d]^{s_{j-1}} \\
     X_{n-2} \ar[r]_-{s_{i-1}} & X_{n-1} 
  }}$$

  $$\vcenter{\xymatrix{
     X_{n-1} \drpullback \ar[r]^{d_i}\ar[d]_{s_j} & X_{n-2} \ar[d]^{s_{j-1}} \\
     X_{n} \ar[r]_{d_i} & X_{n-1} 
  }}  \qquad
  \qquad
  \vcenter{\xymatrix{
     X_{n-1} \drpullback \ar[r]^{d_{j-1}}\ar[d]_{s_{i-1}} & X_{n-2} \ar[d]^{s_{i-1}} \\
     X_{n} \ar[r]_{d_j} & X_{n-1} .
  }}$$
\end{lemma}

\begin{proof}
  We do the first square;  the others are very similar.
    In the composite square
    $$\xymatrix{
     X_{n+1} \ar[r]^{d_i}\ar[d]_{d_{j+1}} & X_{n} \ar[d]^{d_j} 
     \ar[r]^{d_0{}^i} & X_{n-i} \ar[d]^{d_{j-i}}\\
     X_{n} \ar[r]_{d_i} & X_{n-1} \ar[r]_{d_0{}^i} & X_{n-1-i}
  }$$
  the right-hand square is a pullback by the decomposition-space axiom.  The
  composite horizontal maps are composites of bottom face maps, since $d_0{}^i
  \circ d_i = d_0{}^{i+1}$.  Therefore also the composite square is a pullback,
  again by the decomposition-space axiom.  But then the left-hand square is a
  pullback by the usual pullback argument of Lemma~\ref{pbk}.
\end{proof}

\begin{blanko}{Remark.}
  Informally, the lemma states that a given degeneracy map $s_i$ forms pullbacks
  against any other face or degeneracy map, except against $d_{i+1}$ (and except
  against itself), and that a given active face map $d_i$ forms pullbacks
  against any other face or degeneracy maps, except against $s_{i-1}$ (and
  except against itself).  The cases excluded will play a role to characterise
  important special classes of decomposition spaces:
  the pullback squares with $s_i$ against itself characterise {\em complete}
  decomposition spaces \cite[\S 2]{GKT:DSIAMI-2}, while the pullback squares
  with $s_i$ against $d_{i+1}$ expresses the property of being {\em split}
  \cite[\S 5]{GKT:DSIAMI-2}.
\end{blanko}

\begin{blanko}{Remark.}
  In $1$-category theory, all commuting squares of codegeneracy maps in 
  $\simplexcategory$ are {\em absolute} pushouts
  (see Joyal--Tierney~\cite[Thm.~1.2.1]{Joyal-Tierney:Q47}), hence in every
   simplicial {\em set} $X$ the squares of Case 2 of Lemma~\ref{bonus-pullbacks} are pullbacks. However, those 
  pushout squares are not absolute in the sense of $\infty$-categories,
  and not all simplicial {\em spaces} $X$ satisfy this condition,
  which is a special feature of decomposition spaces.
\end{blanko}
\section{\culf functors and decalage}

\label{sec:cULF}

A simplicial map $F: Y \to X$
is called {\em ULF (unique lifting of factorisations)}
if it is a cartesian natural transformation on each active coface
map of $\simplexcategory$.  It is called {\em conservative} if it is cartesian
on each codegeneracy map. It is called {\em \culf} if it is both conservative and ULF. 

\begin{lemma}\label{lem:simp-culf}
  For a simplicial map $F: Y \to X$,
  the following are equivalent.
  \begin{enumerate}
  \item $F$ is cartesian on each inner coface map and on each codegeneracy map (i.e.~\culf).
  \item $F$ is cartesian on each active map of the form $[1]\to [n]$.
  \item $F$ is cartesian on all active maps.
  \end{enumerate}
\end{lemma}

\begin{proof}
  The implication $(1)\Rightarrow (2)$ is easy
  since the active map $[1]\to [n]$ factors as a sequence of inner coface maps
  (or is a codegeneracy map if $n=0$).
  For the implication $(2) \Rightarrow (3)$, consider 
  a general active map $[n]\to [m]$, and
  observe that if $F$ is cartesian on the composite of
  active maps $[1]\to[n]\to[m]$ and also on the active
  map $[1]\to[n]$, then it is cartesian on $[n]\to[m]$ also, by Lemma~\ref{pbk}.
  The implication $(3) \Rightarrow (1)$ is trivial.
\end{proof}

\begin{prop}\label{prop:ULF=>cons}
  Any ULF map  between decomposition spaces is conservative also.
\end{prop}
\begin{proof}
  By Lemma~\ref{lem:simp-culf}(2) it is enough to prove that $F$ is cartesian on
  active maps of the form $[1]\to[n]$.  Since $F$ is ULF, we already know it is 
  cartesian on $[1]\to[n]$ for $n\geq1$, so it remains to check the map 
    $s^0:[1]\to[0]$.
  In the diagram 
  $$
  \xymatrix  {
  Y_0 \ar[rrd]^(0.7){s_0}
  \ar[r]^{s_0} \ar[d] & Y_1 \ar[rrd]^(0.7){s_1}\ar[d]|\hole&& \\
  X_0 \ar[rrd]_(0.4){s_0} \ar[r]^{s_0} & X_1 \ar[rrd]|\hole^(0.7){s_1} & Y_1  \ar[r]_{s_0}  
  \ar[d] & Y_2 \ar[d] \drpullback \ar@{-->}[r]^{d_1} & Y_1 \ar@{-->}[d]\\
  && X_1 \ar[r]_{s_0} & X_2 \ar@{-->}[r]_{d_1} & X_1
  }
$$
the front square is a pullback since it is a section to the dashed square, which
is a pullback since $F$ is ULF. The top and bottom faces of the cube are
pullbacks by Lemma~\ref{bonus-pullbacks}, so the back face is a pullback
square by the basic Lemma~\ref{pbk}.
\end{proof}
\begin{lemma}\label{lem:cULFs0d1}
  A simplicial map between decomposition spaces 
  is \culf if and only if
  it is cartesian on the active map $[1]\to[2]$.
\end{lemma}%
\begin{proof}
    Suppose $X$, $Y$ are decomposition spaces.
    By Lemma~\ref{lem:s0d1} all active face maps in $X$, $Y$ are pullbacks of $d_1:X_2\to X_1$, $d_1:Y_2\to Y_1$.
    If $F:Y\to X$ is cartesian on the active map $[1]\to[2]$ it therefore follows by a basic pullback argument that it is cartesian on all active maps of $\unDelta$. 
The map $F$ is thus ULF, and hence is \culf by Proposition~\ref{prop:ULF=>cons}.
\end{proof}

\begin{blanko}{Remark.}
  The notion of \culf
  can be seen as an abstraction of coalgebra homomorphism, 
  cf.~\ref{lem:coalg-homo} below: `conservative' corresponds to counit preservation,
  `ULF' corresponds to comultiplicativity.
  
  In the special case where $X$ and $Y$ are fat nerves of $1$-categories, 
  then the condition that the square
  $$
  \xymatrix{
  Y_0 \ar[d] \ar[r] \drpullback & Y_1 \ar[d] \\
  X_0\ar[r] & X_1}$$
  be a pullback is precisely the classical notion of conservative functor
  (i.e.~if $f(a)$ is invertible then already $a$ is invertible).
 
  Similarly, the condition that the square
  $$\xymatrix{
  Y_1 \ar[d] & \ar[l] \dlpullback Y_2 \ar[d] \\
  X_1 & \ar[l] X_2}
  $$
  be a pullback is an up-to-isomorphism version of the classical notion of ULF
  functor, implicit already in Content--Lemay--Leroux~\cite{Content-Lemay-Leroux},
  and perhaps made explicit first by Lawvere~\cite{Lawvere:statecats};
  it is equivalent to the notion of discrete Conduch\'e 
  fibration~\cite{Johnstone:Conduche'}. See
  Street~\cite{Street:categorical-structures} for the $2$-categorical notion.
  In the case of the \M categories of Leroux, where there are no
  invertible arrows around, the two notions of ULF coincide.
\end{blanko}

\begin{blanko}{Example.}\label{ex:SIOI}
  Here is an example of a functor which is not \culf in Lawvere's sense
  (is not \culf on classical nerves), but which is \culf
  in the homotopical sense, on fat nerves.  Namely, let $\kat{OI}$ denote the category of 
  finite ordered sets and monotone injections.  Let $\kat{I}$ denote the category of 
  finite sets and injections.  The forgetful functor $\kat{OI} \to \kat{I}$
  is not \culf in the classical sense, because the identity monotone map
  $\un 2 \to \un 2$ admits a factorisation in $\kat{I}$ that does not lift
  to $\kat{OI}$, namely the factorisation into two nontrivial transpositions.
  However, it is \culf in our sense, as can easily be verified by checking
  that the square
  $$\xymatrix{
     \kat{OI}_1 \ar[d] & \ar[l]\dlpullback\kat{OI}_2 \ar[d] \\
     \kat{I}_1  & \ar[l]\kat{I}_2
  }$$
  is a pullback of groupoids, by computing the fibres of the horizontal maps
  over a given monotone injection.
\end{blanko}

\begin{lemma}\label{cULF/decomp}
  If $X$ is a decomposition space and $F: Y \to X$ is \culf
  then also $Y$ is a decomposition space.
\end{lemma}

\begin{blanko}{Decalage.}\label{Dec}
  (See Illusie~\cite[VI.1]{Illusie2}).
  Given a simplicial space $X$
  (as in the top row of the following diagram)
  the {\em lower dec} $\Decbot{X}$ is a new simplicial
  space (the bottom row of the diagram) obtained by deleting $X_0$ and shifting
  everything one place down, deleting also all $d_0$ face maps and all $s_0$
  degeneracy maps.  It comes equipped with a simplicial map, which we
  call the {\em dec map}, $d_\bot:\Decbot{X}\to X$ given by the original
  $d_0$:
$$
\xymatrix@C+1em{
X&&X_0  
\ar[r]|(0.55){s_0} 
&
\ar[l]<+2mm>^{d_0}\ar[l]<-2mm>_{d_1} 
X_1  
\ar[r]<-2mm>|(0.6){s_0}\ar[r]<+2mm>|(0.6){s_1}  
&
\ar[l]<+4mm>^(0.6){d_0}\ar[l]|(0.6){d_1}\ar[l]<-4mm>_(0.6){d_2}
X_2 
\ar[r]<-4mm>|(0.6){s_0}\ar[r]|(0.6){s_1}\ar[r]<+4mm>|(0.6){s_2}  
&
\ar[l]<+6mm>^(0.6){d_0}\ar[l]<+2mm>|(0.6){d_1}\ar[l]<-2mm>|(0.6){d_2}\ar[l]<-6mm>_(0.6){d_3}
X_3 
\ar@{}|\cdots[r]
&
\\
\\
\Decbot{X}\ar[uu]_{d_\bot}&&X_1  \ar[uu]_{d_0}
\ar[r]|(0.55){s_1} 
&
\ar[l]<+2mm>^{d_1}\ar[l]<-2mm>_{d_2} 
X_2  \ar[uu]_{d_0}
\ar[r]<-2mm>|(0.6){s_1}\ar[r]<+2mm>|(0.6){s_2}  
&
\ar[l]<+4mm>^(0.6){d_1}\ar[l]|(0.6){d_2}\ar[l]<-4mm>_(0.6){d_3}
X_3 \ar[uu]_{d_0}
\ar[r]<-4mm>|(0.6){s_1}\ar[r]|(0.6){s_2}\ar[r]<+4mm>|(0.6){s_3}  
&
\ar[l]<+6mm>^(0.6){d_1}\ar[l]<+2mm>|(0.6){d_2}\ar[l]<-2mm>|(0.6){d_3}\ar[l]<-6mm>_(0.6){d_4}
X_4 \ar[uu]_{d_0}
\ar@{}|\cdots[r]
&
}
$$
In fact $\Decbot{}$ is a comonad on simplicial spaces, with the dec map
$d_\bot$ as its counit.

Similarly the \emph{upper dec} $\Dectop{X}$ is obtained by instead 
deleting, in each degree, the last face map $d_\top$ and the last degeneracy map 
$s_\top$. The deleted last face map becomes the \emph{dec map} 
$d_\top \colon \Dectop{X} \to X$.
\end{blanko}

\begin{blanko}{Slice interpretation.}\label{names-and-slices}
  If $X=N\CC$ is the strict nerve of a category $\CC$
  then there is a close relationship between the upper dec and the slice 
  construction: $\Dectop{X}$ is the disjoint
  union of all (the nerves of) the slice categories of $\CC$:
  $$
  \Dectop{X} = \sum_{x\in X_0} N(\CC_{/x}).
  $$
  (In general it is a homotopy sum.)
  
  Any individual slice category
  can be extracted from the upper dec, by exploiting that the upper dec
  comes with a canonical augmentation given by (iterating) the bottom face map.
  The slices are the fibres of this augmentation:
  $$
  \xymatrix{
    N\CC_{/x}\rto\drpullback\dto&\Dectop{X}\dto^{d_\bot}\\1\rto_{\name x}&X_0\,.
  }
  $$
  
  There is a similar relationship between the lower dec and the coslices.
\end{blanko}

\begin{prop}\label{Dec=Segal+cULF}
If $X$ is a decomposition space then 
$\Dectop{X}$ and $\Decbot{X}$ 
are Segal spaces, and the dec maps 
$d_\top : \Dectop{X} \to X$ and  $d_\bot : \Decbot{X} \to X$ 
are \culf.
\end{prop}
\begin{proof}
    We put $Y=\Dectop{X}$ and check the pullback 
    condition \ref{segalpqr} (3),
        $$
\vcenter{\xymatrix{
   Y_{n+1}\dto_{{d_\top}}\drpullback 
\rto^-{d_\bot}
 &  Y_n\dto^{d_\top} \\
   Y_{n}\rto_-{d_\bot}  &  Y_{n-1}
  .}}
$$
This is the same as
    $$
\vcenter{\xymatrix{
   X_{n+2}\dto_{{d_{\top-1}}}\drpullback 
\rto^-{d_\bot}
 &  X_{n+1}\dto^{d_{\top-1}} \\
   X_{n+1}\rto_-{d_\bot}  &  X_{n} ,
  }}
$$
and since here the vertically drawn maps (which with respect to $Y$ are 
outer face maps) are inner face maps in $X$, this pullback square is one
of the decomposition-space axioms.
The \culf conditions say that the various $d_\top$ form pullbacks with all
active maps in $X$.  But this follows from the decomposition-space axiom for 
$X$.
\end{proof}

\begin{thm}\label{thm:decomp-dec-segal}
For a simplicial space $X: \simplexcategory\op\to\Grpd$, the following are equivalent
\begin{enumerate}
\item $X$ is a decomposition space
\item  both $\Dectop{X}$ and $\Decbot{X}$ are 
    Segal spaces, and the respective dec maps back to $X$ are \culf. 
\item  both $\Dectop{X}$ and $\Decbot{X}$ are 
    Segal spaces, and the respective dec maps back to $X$ are conservative.
\item  both $\Dectop{X}$ and $\Decbot{X}$ are Segal spaces, and 
the following squares are pullbacks:
$$  
\xymatrix{
   X_1\rto^{s_1}\drpullback 
\dto_-{d_\bot}
 & X_2\dto^{d_\bot} \\
  X_0  \rto_-{s_0}  &  X_1,
}\qquad
\xymatrix{
   X_1\dto_{d_\top}\drpullback 
\rto^-{s_0}
 &  X_2 \dto^{d_\top} \\
X_0  \rto_-{s_0}  &  X_1.
  }
$$
\end{enumerate}
\end{thm}

\begin{proof} 
The implication (1) $\Rightarrow$ (2) is just the preceding Proposition, and the implications (2) $\Rightarrow$ (3) $\Rightarrow$ (4) are specialisations.
The implication (4) $\Rightarrow$ (1) follows from Proposition~\ref{onlyfourdiags}.
\end{proof}

\begin{blanko}{Remark.}
  Dyckerhoff and Kapranov~\cite{Dyckerhoff-Kapranov:1212.3563} (Theorem~6.3.2)
  obtain the result that a simplicial space is $2$-Segal (i.e.~a decomposition
  space except that there are no conditions imposed on degeneracy maps)
  if and only if both $\operatorname{Dec}$s are Segal spaces.
\end{blanko}

\begin{blanko}{Right and left fibrations.}\label{rightfibSegal}
  A simplicial map $F: Y \to X$ is called a {\em right
  fibration} if it is cartesian on all bottom face maps $d_\bot$.  This implies
  that it is also cartesian on all active maps (i.e.~is \culf), as follows from
  an easy argument with the basic pullback Lemma~\ref{pbk}.  The terminology is
  motivated by the case where $Y$ and $X$ are Segal spaces, in which case it
  corresponds to standard usage in the theory of $\infty$-categories.
  If $X$ and $Y$ are fat nerves of $1$-categories, then `right fibration' corresponds to groupoid fibration in the sense of Street~\cite{Street:fibsinbicats}. 
  Similarly, $F$ is called a {\em left fibration} if it is cartesian on $d_\top$
  (and consequently on all active maps also).
\end{blanko}

\begin{prop}\label{DecULF-is-right}
  If $F:Y \to X$ is a \culf functor then $\Decbot{(F)} : \Decbot{Y} \to
  \Decbot{X}$ is a right fibration.  Similarly, $\Dectop{(F)} : \Dectop{Y}
  \to \Dectop{X}$ is a left fibration.
\end{prop}
\begin{proof}
  It is clear that if $F$ is \culf then so is $\Decbot{(F)}$.  The further claim
  is that $\Decbot{(F)}$ is also cartesian on $d_0$.  But $d_0$ was originally
  a $d_1$, and in particular was active, hence  $\Decbot{(F)}$ is cartesian on this map.
\end{proof}
 
\section{The incidence coalgebra of a decomposition space}
\label{sec:COALG}

We now turn to the incidence coalgebra (with $\infty$-groupoid coefficients)
associated to any decomposition space, and explain
the origin of the decomposition-space axioms.

The \emph{incidence coalgebra} of a decomposition space $X$ will be a
comonoid object in the symmetric monoidal $\infty$-category $\LIN$,
and the underlying object is $\Grpd_{/X_1}$.
Since $\Grpd_{/X_1} \tensor \Grpd_{/X_1} = \Grpd_{/X_1\times X_1}$,
and since linear functors are given by spans, to define a comultiplication
functor is to give a span
$$
X_1 \leftarrow M \to X_1 \times X_1\,.
$$

For any simplicial space $X$, we can consider the following
structure maps on $\Grpd_{/X_1}$.
\begin{blanko}{Comultiplication and counit.}\label{comult}
  The span
  $$
  \xymatrix{
   X_1  & \ar[l]_{d_1}  X_2\ar[r]^-{(d_2,d_0)} &  X_1\times X_1 
  }
  $$
  defines a linear functor, the {\em comultiplication}
  \begin{eqnarray*}
    \Delta : \Grpd_{/ X_1} & \longrightarrow & 
    \Grpd_{/( X_1\times X_1)}  \\
    (A\stackrel a\to X_1) & \longmapsto & (d_2,d_0){} \lowershriek \circ {d_1} \upperstar(a) .
  \end{eqnarray*}
Likewise, the span
  $$
  \xymatrix{
  X_1  & \ar[l]_{s_0}  X_0\ar[r]^{t} &  1 
  }
  $$
  defines a linear functor, the {\em counit}
  \begin{eqnarray*}
    \varepsilon : \Grpd_{/ X_1} & \longrightarrow & 
    \Grpd  \\
        (A\stackrel a\to X_1) & \longmapsto & t {}\lowershriek \circ {s_0} \upperstar(a) .
  \end{eqnarray*}

If $X$ is the nerve of a category (for example, a poset) then $X_2$ is the set 
of all composable pairs of arrows. The comultiplication is just the 
formula~\eqref{Delta-f} from the introduction
$$
\Delta(f)\:=\sum_{b\circ a=f} a\otimes b,
$$
and the counit is the classical counit, sending identity arrows to $1$ and 
other arrows to $0$.
\end{blanko}

  \begin{blanko}{Running example: the   Hopf algebra of rooted trees.}\label{ex:trees-coalg}
Recall from     Example~\ref{ex:trees-decomp} that $\dstrees$ is the decomposition space in which $\dstrees_1$ is the groupoid of
  rooted forests and $\dstrees_2$ is the groupoid of rooted forests with an admissible cut.
  Taking pullback along $d_1:\dstrees_2 \to \dstrees_1$ is to consider all possible
  admissible cuts $c$ of a given forest, and taking lowershriek along
  $(d_2,d_0): \dstrees_2 \to \dstrees_1 \times \dstrees_1$ is to return the forests $P_c$ and $R_c$ found in the two 
  layers on either side of the cut. We thus have a comultiplication functor
  \begin{align*}\Grpd_{/\dstrees_1} &\longrightarrow \Grpd_{/\dstrees_1} \otimes \Grpd_{/\dstrees_1}
    \\
    (\name T:1\to \dstrees_1)&\longmapsto (\name{P_c}:1\to \dstrees_1)\otimes (\name{R_c}:1\to \dstrees_1)\end{align*}
  which is an objective version of the Butcher--Connes-Kreimer comultiplication,
  cf.~\ref{ex:preintro}\eqref{eq:preintro}.
\end{blanko}

\begin{blanko}{Coassociativity of the comultiplication.}\label{coasscircle1}
The desired coassociativity diagram (which should commute up to equivalence)
$$
\xymatrix@!C=25ex@R-1.5ex{
\Grpd_{/X_1}\dto_\Delta\rto^-\Delta&\Grpd_{/X_1\times X_1}\dto^{\Delta\otimes \id}\\
\Grpd_{/X_1\times X_1}\rto_-{\id\otimes\Delta}&\Grpd_{/X_1\times X_1\times X_1}}
$$
is a diagram of linear functors defined by the spans in the \emph{outline} of the following diagram.
$$\xymatrix@C+6ex@R+0ex{
X_1                  & X_2 \ar[l]_{d_1}\ar[r]^{(d_2,d_0)}   & X_1 \times X_1 \\
X_2 \ar[u]^{d_1} \ar[d]_{(d_2,d_0)}&
X_3
   \ar@{}[ur]|*+[o][F-]{1}
   \ar@{}[dl]|*+[o][F-]{2}
\dlpullback\urpullback\ar[l]_{d_2}\ar[u]^{d_1}\ar[d]^{(d_2^2,d_0)}\ar[r]_{(d_3,d_0 d_0)}
          & X_2 \times X_1 \ar[u]_{d_1\times \id}\ar[d]^{(d_2,d_0)\times \id}\\
X_1\times X_1 & X_1 \times X_2 \ar[l]^-{\id\times d_1}
\ar[r]_-{\id\times (d_2,d_0)} &X_1 \times X_1 \times X_1
}
$$
Coassociativity will follow from Beck--Chevalley equivalences if the \emph{interior} 
part of the diagram can be established, with pullbacks $\xymatrix{*+[o][F-]{1}}$, $\xymatrix{*+[o][F-]{2}}$
  as indicated.
Now the upper right-hand square $\xymatrix{*+[o][F-]{1}}$, for example, will be a pullback if and only if
its composite with the first projection is a pullback:
$$\xymatrix@C+6ex@R+0ex{
X_2 \ar[r]^-{(d_2,d_0)}   & X_1 \times X_1  \ar[r]^-{\text{pr}_1} & X_1 \\
X_3 
\ar@{}[ur]|*+[o][F-]{1}
\ar[u]^{d_1}\ar[r]_{(d_3,d_0 d_0)}
  & X_2 \times X_1 \urpullback \ar[u]^{d_1\times \id} \ar[r]_-{\text{pr}_1} & X_2 
  \ar[u]_{d_1}
}
$$
But demanding the outer rectangle to be a pullback
is precisely one of the basic decomposition-space axioms.
This argument is the origin of the decomposition-space axioms.

If one is just interested in coassociativity at the level of $\pi_0$ of the incidence coalgebra,
this square and its twin are all that are needed.
This was the case in the work of To\"en~\cite{Toen:0501343} who dealt 
with the case where $X$ is the Waldhausen \Sdot-construction of
a dg category, and in the work of Dyckerhoff and
Kapranov~\cite{Dyckerhoff-Kapranov:1212.3563} for exact $\infty$-categories.
\end{blanko}

\begin{blanko}{Homotopy coherence of coassociativity.}
  For coassociativity of the incidence coalgebra
  at the objective level, higher coherence has to be
  established, which will require the full decomposition-space axioms. To 
  establish coassociativity in a strong homotopy sense we must deal on
  an equal footing with all `reasonable' spans
  \begin{equation}\label{pre-reasonable}
  \prod X_{n_j}\leftarrow \prod X_{m_j}\rightarrow\prod X_{k_i}
  \end{equation}
  which could arise from composites of products of the comultiplication and counit.
  We therefore take a more
  abstract approach, relying on some more simplicial machinery.
\end{blanko}

\section{Decomposition spaces as monoidal functors}
\label{sec:DD}

In this section, in order to establish the homotopy coherent coassociativity of
the incidence coalgebra, we study the twisted arrow category $\DD$ of the
category of finite ordinals, with a certain external tensor product $\oplus$.
In Proposition~\ref{DD-universal} we show that simplicial objects in a cartesian
monoidal category correspond to monoidal functors from $\DD$, which enables us
to characterise decomposition spaces as monoidal functors $X: (\DD,\oplus) \to
(\Grpd,\times)$ satisfying an exactness condition.  The purpose of this
viewpoint is to deal with products of the form $\prod X_{k_i}$, as they appear
in the `reasonable spans' \eqref{pre-reasonable}, to which we return in
\S\ref{sec:proofofcoass}.

\begin{blanko}{The category $\unDelta$ of finite ordinals (the algebraist's Delta).} 
  We denote by $\unDelta$ the category of all finite ordinals (including the
  empty ordinal) and monotone maps.  Clearly $\simplexcategory \subset 
  \unDelta$,
  but this is not the most useful relationship between the two 
  categories, and
  we will use a different notation for the objects of $\unDelta$, 
  given by their cardinality, with an underline:
  $$
  \un n = \{1,2,\ldots,n\} .
  $$
  The category $\unDelta$ is monoidal under ordinal sum
  $$
  \un m + \un n := \un{m{+}n} ,
  $$
  with $\un 0$ as the neutral object.
\end{blanko}

Recall that $\Deltaact$ is the subcategory of $\simplexcategory$ containing only
the active maps, and that it is a monoidal category under amalgamated ordinal
sum $\intconcat$ (cf.~\ref{bla:pm}).

\begin{lemma}\label{lem:Delta-duality}
There is a canonical equivalence of monoidal categories (an isomorphism, if
we consider the usual skeleta of these categories)
\begin{eqnarray*}
(\unDelta, +, \un 0) &\simeq& (\Deltaact\op, \intconcat, [0]) \\
  \un k & \leftrightarrow & [k]
\end{eqnarray*}
\end{lemma}

\begin{proof}
  The map from left to right sends $\un k \in \unDelta$ to 
  $$
  \Hom_{\unDelta}(\un k, \un 2) \simeq [k] \in \Deltaact\op .
  $$
  The map in the other direction sends 
  $[k]$ to the ordinal
  $$
  \Hom_{\Deltaact}([k],[1]) \simeq \un k .
  $$
  In both cases, functoriality is given by precomposition.
\end{proof}

In both categories we can picture the objects as a line with some dots.
The dots then represent the elements in $\un k$, while the edges represent
the elements in $[k]$;  a map operates on the dots when considered a
map in $\unDelta$ while it operates on the edges when considered a map in 
$\Deltaact$.
Here is a picture of a certain map $\un 5 \to \un 4$ in
$\unDelta$ and of the corresponding map $[5] \leftarrow [4]$ in
$\Deltaact$.
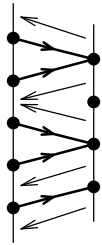
\begin{center}
\begin{texdraw}
  \arrowheadtype t:V
  
    \arrowheadsize l:6 w:3  

  \move (0 0) 
  \bsegment
  \move (0 -5)
  \lvec (0 85)
  \move (0 8) \Onedot 
  \move (0 24) \Onedot 
  \move (0 40) \Onedot 
  \move (0 56) \Onedot 
  \move (0 72) \Onedot 
  \esegment
  
  \move (30 0)
    \bsegment
    \move (0 3)
  \lvec (0 77)
  \move (0 16) \Onedot 
  \move (0 32) \Onedot 
  \move (0 48) \Onedot 
  \move (0 64) \Onedot 
  \esegment

  \move   (27 8) \avec (3 0)
  \move  (27 72) \avec (3 80) 
    \move   (27 55) \avec (3 49)
    \move   (27 41) \avec (3 47)
    \move   (27 23.7) \avec (3 16.3)

        \arrowheadsize l:4 w:3  

    \linewd 0.9
    \move (0 8) \avec (15 12)\lvec (30 16)
    \move (0 24)\avec (15 28)\lvec (30 32)
    \move (0 40)\avec (15 36)\lvec (30 32)
    \move (0 56)\avec (15 60)\lvec (30 64)
    \move (0 72)\avec (15 68)\lvec (30 64)
\end{texdraw}
\end{center}

\begin{blanko}{A twisted arrow category of $\unDelta$.}
  Consider the category $\DD$ whose objects are the arrows
  $\un n \to \un k$ of $\unDelta$ 
  and whose morphisms $(g,f)$ from $a:\un m \to \un h$ to $b:\un n \to \un
  k$ are commutative squares
  \begin{align}\label{mor-DD}
  \vcenter{\xymatrix{
  \un m \ar[d]_a \ar[r]^g \ar@{}[rd]|{(g,f)}& \un n \ar[d]^b \\
  \un h &  \ar[l]^f \un k .
  }}\end{align}
  That is, $\DD\op$ is the twisted arrow category~\cite{MacLane:categories,BaWi:1985}
  of $\unDelta$. 
\end{blanko}

\begin{blanko}{Factorisation system on $\DD$.}
  There is a canonical factorisation system on $\DD$: any morphism
  \eqref{mor-DD} factors uniquely as
  $$\xymatrix@C=3.2pc{
  \un m \ar[d]_{a=fbg} \ar[r]^= \ar@{}[rd]|\varphi
  &\un m \ar[d]|{bg} \ar[r]^g \ar@{}[rd]|\gamma
  & \un n \ar[d]^b \\
  \un h &  \ar[l]^f \un k&  \ar[l]^= \un k  .
  }$$
  The maps $\varphi=(\id,f):fb\to b$ in the left-hand class of the
  factorisation system are termed \emph{segalic},
  \begin{align}\label{cov-DD}\vcenter{\xymatrix{
  \un m \ar[r]^=  \ar[d]_{fb}\ar@{}[rd]|\varphi
  & \un m \ar[d]^b \\
  \un h & \ar[l]^f \un k.}}
  \end{align}
  The maps $\gamma=(g,\id):bg\to b$ in the right-hand class are termed {\em
  ordinalic} and may be identified with maps in various slice categories
  $\unDelta_{/\un k}$ \begin{align}\label{gen-DD}\vcenter{\xymatrix{ \un m
  \ar[r]^g \ar[d]_{bg}\ar@{}[rd]|\gamma & \un n \ar[d]^{b} \\
  \un k & \ar[l]^= \un k.}}
  \end{align}
\end{blanko}

\begin{blanko}{External sum.}
  Observe that $\unDelta$ is isomorphic to the subcategory of objects with
  target $\un k=\un 1$, termed the {\em connected objects} of $\DD$,
  \begin{align}\label{unD-in-DD}
  \unDelta\xrightarrow{\;\;=\;\;}\unDelta_{/\un1}\xrightarrow{\;\;\subseteq\;\;} \DD.\end{align} 

  The ordinal sum operation in $\unDelta$ induces a monoidal operation in
  $\DD$: the {\em external sum} $(\un n{\to}\un k)\oplus(\un n'{\to}\un k')$
  of objects in $\DD$ is their ordinal sum $\un n+\un n'\to \un k+\un k'$
  as morphisms in $\unDelta$.  The neutral object is $\un 0 \to \un 0$.
  The inclusion functor \eqref{unD-in-DD} is not monoidal, but it is easily
  seen to be oplax monoidal by means of the codiagonal map $\un1+\un1\to
  \un1$.

  Each object $\un m\xrightarrow{\;a\;}\un k$ of $\DD$ is an external sum
  of connected objects,
  \begin{align}\label{unD-DD}
  a\;\;=\;\;a_1\oplus a_2\oplus \dots \oplus a_k\;\;=\;\;\bigoplus_{i\in \un k}
  \left(\un m_i\xrightarrow{\;a_i\;}\un1\right),
  \end{align}
  where $\un m_i$ is (the cardinality of) the fibre of $a$ over $i\in\un k$. 

Any segalic map \eqref{cov-DD} and any ordinalic map \eqref{gen-DD} in $\DD$ may be written uniquely as external sums 
\begin{align}\label{cov-sum}
 \varphi &\;\;= \;\;\varphi_1 \oplus \varphi_2\oplus
\dots \oplus \varphi_h \;\;=\;\;\bigoplus_{j\in \un h}
\left(\vcenter{\xymatrix@R1.4pc{
\un m_j\rto^=\dto
\ar@{}[rd]|{\varphi_j}
&\dto^{b_j} \un m_j\\  \un 1 &\lto \un k_j }}\right)
\\\label{gen-sum}
 \gamma &\;\;= \;\;\gamma_1 \oplus \gamma_2\oplus 
\dots \oplus \gamma_k \;\;=\;\;\bigoplus_{i\in \un k}\left(\un m_i\xrightarrow{\;\gamma_i\;}\un n_i\right)
\end{align}
where each $\gamma_i$ is a map in  $\unDelta_{/\un1}=\unDelta$.
\end{blanko}

In fact $\DD$ is a universal monoidal category in the following sense.
\begin{prop}\label{DD-universal}
For any cartesian category $(\CC,\times,1)$, there is an equivalence
$$
\Fun(\simplexcategory\op,\CC)\;\simeq \;\Fun^\otimes((\DD,\oplus,0),(\CC,\times,1))
$$
between the categories of simplicial objects $X$ in $\CC$ and of
monoidal functors $\overline X:\DD\to\CC$.
The correspondence between $X$ and $\ov X$ is determined by following properties.

(a) The functors $X:\simplexcategory\op\to\CC$ and $\ov X:\DD\to\CC$ agree on the common subcategory $\Deltaact\op\cong\unDelta$,
$$
\xymatrixrowsep{8pt}
\xymatrix@C+1pc{
\ar[dd]_{\cong}
\Deltaact\op\,\ar@{^(->}[r]&\simplexcategory\op\ar[rd]^X\\
&&\CC .\\
\unDelta\;\ar@{^(->}[r]
&\DD\ar[ru]_{\overline X}& }
$$

(b) Let $(\un m\stackrel a\to \un k)=\bigoplus_i (\un m_i\xrightarrow {a_i} \un 1)$ be the external sum decomposition \eqref{unD-DD} of any object of $\DD$, and denote by
$f_i:[m_i]\rat[m_1]\intconcat\dots\intconcat[m_k]=[m]$ the canonical inert map in $\simplexcategory$, for $i\in\un k$. Then
$$\ov X\!\left(\!\vcenter{\xymatrix@R1.1pc@C1.6pc{\un m\rto^=\dto
\ar@{}[rd]|\varphi
&\dto^a \un m\\  \un 1 &\lto \un k }}
\right)\;=\;(X(f_1),\dots,X(f_k))\;:\;
X_m\longrightarrow\prod_{i\in\un k} X_{m_i}
$$
and each $X(f_i)$ is the composite of $\ov X(\varphi)$ with the projection to $X_i$.
\end{prop}
\begin{proof}
Given $\ov X$, property (a) says that there is a unique way to define 
$X$ on objects and active maps.
Conversely, given $X$, 
then for any object $a:\un m\to\un k$ in $\DD$ we have  
$$
\overline X_a\;\;=\;\;
\prod_{i\in\un k}
\overline X_{a_i}\;\;=\;\;\prod_{i\in\un k}X_{m_i}
$$
using \eqref{unD-DD}, and for any ordinalic map $\gamma$ we have
$$
\overline X(\gamma)\;\;=\;\;
\prod_{i\in\un k}
\overline X(\gamma_i)\;\;=\;\;\prod_{i\in\un k}X(g_i)
$$
using \eqref{gen-sum}, where $g_i\in\Deltaact\op$ corresponds to $\gamma_i\in\unDelta$. 

Thus we have a bijection between functors $X$ defined on $\Deltaact\op$ and monoidal functors $\ov X$ defined on the ordinalic subcategory of $\DD$.
Now we consider the inert and segalic maps. Given $\ov X$, 
property (b) says that for any inert map $f_r:[m_r]\to[m]$ we may define 
$$X(f_r)=
\left(X_m\xrightarrow{\overline X(\varphi)}\prod_{i\in\un k} X_{m_i}\onto X_{m_r}\right)$$
We may
assume $k=3$:
given the factorisation
$$
\varphi=\left(\vcenter{\xymatrix@R1.5pc@C1.6pc{
\un m\rto^-=
\dto\ar@{}[rd]|{\varphi_2}
&\dto \un m_{<r}+\un m_r+\un m_{>r}\rto^-=
\dto\ar@{}[rd]|{\varphi_1\oplus\id\oplus \varphi_3}
&\dto\sum_{i\in \un k} m_i
\\  \un 1 &\lto \un 3&\lto\un k }}\right)
$$
one sees the value $X(f_r)$ is well defined from the following diagram 
$$
\xymatrix@C=6pc{X_m\rto^-{\overline X(\varphi_2)}\ar@/_0.9pc/[rrd]_(0.3){X(f_r)}
  &X_{m_{<r}}\times X_{m_r}\times X_{m_{>r}}
   \rto^-{\overline X(\varphi_1)\times\id\times\overline X(\varphi_3)}
   \ar@{->>}[dr]
       &\prod_{i\in\un k} X_{m_i}
        \ar@{->>}[d]
\\ &   &X_{m_r}\;.
}
$$
Functoriality of $X$ on a composite of inert maps, say $[m_3]\rat[\sum_2^4m_i]\rat[\sum_1^5m_i]$,
now follows from the diagram
$$
  \xymatrixrowsep{16pt}
  \xymatrixcolsep{16pt}
\xymatrix {
X_{\sum_1^5m_i}\drto\rrto &&\ar@{->>}[dr] \prod_1^5X_{m_i}\ar@{->>}[rr]&&X_{m_3}\\
&\!\!\!X_{m_1}\times X_{\sum_2^4m_i}\times X_{m_5}\!\!\!
\ar[ur]\ar@{->>}[dr]&&\prod_2^4X_{m_i}\ar@{->>}[ur]\\
&&X_{\sum_2^4m_i}\urto
}
$$
in which the first triangle commutes by functoriality of $\ov X$.

Conversely, given $X$, property (b) says how to define $\ov X$ on segalic maps with connected domain and hence, by \eqref{cov-sum}, on all segalic maps.
Functoriality of $\ov X$ on a composite of segalic maps, say $(\id,\un 1\leftarrow\un h\leftarrow\un k)$, follows from functoriality of $X$:
$$
\xymatrix@C=9.5pc{
X_m
\ar@/_2.5pc/[rr]_-{(X([m_i]\rat[m]))_{i\in\un k}}
\rto^-{(X([m_j]\rat[m]))_{j\in\un h}}
&\displaystyle\prod_{j\in \un h}X_{m_j}
\rto^-{\prod_{j\in\un h}(X([m_i]\rat[m_j]))_{i\in\un k_j}}
&\displaystyle\prod_{j\in \un h}\prod_{i\in \un k_j}X_{m_i}
}
$$
It remains only to check that the construction of $\overline X$ from $X$
(and of $X$ from $\overline X$) is well defined on composites of ordinalic
followed by segalic (inert followed by active) maps.  One then has the
mutually inverse equivalences required.  Consider the factorisations in
$\DD$,
$$
\vcenter{\xymatrix{
\un m \ar[d] \ar[r]^= \ar@{}[rd]|{\varphi}&\un m \ar[d] \ar[r]^g \ar@{}[rd]|{\gamma}& \un n \ar[d] \\
\un 1 &  \ar[l] \un k&  \ar[l]^= \un k
  }}
\quad=\quad
\vcenter{\xymatrix{
\un m \ar[d] \ar[r]^g \ar@{}[rd]|{\gamma'}&\un n \ar[d] \ar[r]^=  \ar@{}[rd]|{\varphi'}&\un n \ar[d] \\
\un 1 &  \ar[l]^= \un 1&  \ar[l] \un k
  .}}
$$
To show that $\overline X$ is well defined,
we must show that the diagrams
$$
\xymatrix@C=5pc{
X_m\rrto^-{\ov X(\varphi)=(X(f_1),\dots,X(f_k))}\dto_{\ov X(\gamma')=X(\tilde g)}&& \prod
X_{m_i}\dto_{\ov X(\gamma')=\prod X(\tilde g_i)}\ar@{->>}[r]&X_{m_r}
\dto^{X(\tilde g_r)}\\
X_{n}\rrto_-{\ov X(\varphi')=(X(f'_1),\dots,X(f'_k))}&&\prod
X_{n_i}\ar@{->>}[r]&X_{n_r},
}
$$
commute for each $r$, where $\tilde g$, $\tilde g_i$ in $\Deltaact$ correspond to $g$, $g_i$ in $\unDelta$. This follows by functoriality of $X$, since $\tilde g$ restricted to $n_r$ is the corestriction of $\tilde g_r$.
Finally we observe that this diagram, with $k=3$ and $r=2$, 
also serves to show that the construction of $X$ from $\ov X$ is well defined on
$$
\xymatrix@C5pc{
{[m_1{+}m_2{+}m_3]} &
 \ar@{ >->}[l]_-{f_2} {[m_2]}\\
{[n_1{+}n_2{+}n_3]} \ar@{>|}[u]^{\tilde g}&\ar@{>|}[u]_{\tilde g_2}
\ar@{ >->}[l]^-{f'_2}{[n_2]}}
$$
\end{proof}

\begin{lemma}\label{lemma:DD-pbk}
  In the category $\DD$, ordinalic and segalic maps admit pullback 
  along each other, and the results are again maps of the same type.
\end{lemma}

\noindent (This is a general fact about opposites of 
twisted arrow categories.)

\begin{proof}
  In the diagram below, the map from $a$ to $b$
  is segalic (given essentially by the bottom map $f$) and the map from
  $a'$ to $b$ is ordinalic (given essentially by the top map $g$):
\begin{equation} \label{eq:pbk-cube}\vcenter{
  \xymatrix@C3.9pc@R1.35pc{
  && \ar@{-->}[lldd]_{g} \un m \ar@{-->}[d] \ar@{-->}[rrdd]^*-[@]=0+!D{=} && \\
  && \un h && \\
  \un n \ar[d]_a \ar[rrdd]  ^*-[@]=0+!D{=} &&&& \un m \ar[lldd]_{g} \ar[d]^{a'} \\
  \un h \ar@{-->}[rruu]|(0.25)\hole^*-[@]=0+!D{=} &&&& \un k \ar@{-->}[lluu]|(0.25)\hole
_f \\
  && \un n \ar[d]_b && \\
  && \ar[lluu]^f \un k \ar[rruu]_*-[@]=0+!U{=} &&
  }}\end{equation}
    To construct the pullback, we are forced to repeat $f$ and $g$, completing 
  the squares with 
  the corresponding identity maps.
  The connecting map in the resulting object is 
$fbg: \un m \to \un h$.
  It is clear from the presence of the four identity maps that this is a 
  pullback.
\end{proof}

\begin{blanko}{Examples.}\label{ex:DD-pbk}
Every segalic-ordinalic pullback is the external sum of
\emph{connected} pullbacks, that is, those segalic-ordinalic pullbacks as above where $h=1$.
A segalic-ordinalic pullback over $b=\id$ is termed a \emph{special} pullback.
Any map $g : \un m \to \un n$ in $\unDelta$ defines canonically a
\emph{connected special} pullback:
  \begin{equation*}
  \vcenter{
  \xymatrix@C3pc@R1.2pc{
  && \ar[lldd]_{g} \un m \ar[d] \ar@{=}[rrdd] && \\
  && \un 1 && \\
  \un n \ar[d] \ar@{=}[rrdd] &&&& \un m \ar[lldd]_{g} \ar[d]^{g} \\
  \un 1 \ar@{=}[rruu]|(0.25)\hole &&&& \un n \ar[lluu]|(0.25)\hole\\
  &&n \ar@{=}[d] && \\
  && \ar[lluu] \un n \ar@{=}[rruu] &&
  }}\end{equation*}
\end{blanko}

We now have the following important characterisation of decomposition
spaces.

  \begin{prop}\label{prop:DDecomp}
  For $X: \simplexcategory\op\to\Grpd$ 
  a simplicial space, the following are equivalent.
  \begin{enumerate}
\item $X$ is a decomposition space.

\item 
  The corresponding monoidal functor $\overline X : \DD \to \Grpd$ preserves
  pullbacks of the kind described in Lemma~\ref{lemma:DD-pbk}.
  
  \item For every active map $g: [n] \genmap [m]$ the following square
  is a pullback
  $$
  \xymatrix{
  X_m \ar[r] \ar[d]_{g^*} & X_{m_1} \times \cdots \times X_{m_n} \ar[d]^{g_1^* \times 
  \cdots \times g_n^*} \\
  X_n \ar[r] & X_1 \times \cdots \times X_1,
  }$$
  where 
  $g = g_1 \vee \dots \vee g_n$ with $g_i:[1]\genmap[m_i]$, and the horizontal maps are induced by
  the  inert maps $[m_i] \rat [m_1]\vee \dots \vee [m_n]=[m]$ and
  $[1] \rat [1]\vee\dots\vee[1]=[n]$.
\end{enumerate}
\end{prop}

\begin{proof}
  Since $\overline X$ is monoidal, condition (2) is equivalent to the
  condition that the connected pullbacks are preserved.  Now the $\overline
  X$-image of a connected pullback is a diagram
  $$
  \xymatrix{
  X_m \ar[r] \ar[d] & X_{m_1} \times \cdots \times X_{m_k} \ar[d] \\
  X_n \ar[r] & X_{n_1} \times \cdots \times X_{n_k}.
  }$$
We can factor this into a vertical composite of such diagrams 
in which the map on the left is a single face or degeneracy map.
Then the map on the right is a product of 
  maps, one of which, say the $i$th factor, is again a single face 
  or degeneracy map, and the rest are identities.
  To check if each of these new simpler squares are pullbacks we consider
  the projections onto the non-trivial factor:
    $$
  \xymatrix{
  X_m \ar[r] \ar[d] & X_{m_1} \times \cdots \times X_{m_k}
  \ar[d] \ar[r] & 
  X_{m_i} \ar[d]\\
  X_n \ar[r] & X_{n_1} \times \cdots \times X_{n_k} 
  \ar[r] & X_{n_i}  .
  }$$
  But by construction of $\overline X$, 
  the composite horizontal maps are precisely inert maps in the sense of the
  simplicial space $X$, and the vertical maps are precisely active maps 
  in the sense that they are arbitrary maps in $\unDelta$ and hence (in the 
  other direction) active maps in $\simplexcategory$, under the duality
  in Lemma~\ref{lem:Delta-duality}.
  Since the right-hand square is always a pullback, the standard pullback
  argument of Lemma~\ref{pbk} shows that
  the total square is a pullback (i.e.~we have a decomposition space)
  if and only if the left-hand square is a pullback (i.e.~the pullback condition on 
  $\overline X$ is satisfied).  This proves $(1)\Leftrightarrow(2)$.

    The diagram in condition (3) is the image of a connected special
  pullback, as in Example~\ref{ex:DD-pbk}, so $(2)\Rightarrow(3)$.
  Finally we show that $(3)\Rightarrow(2)$. As $\overline X$ is monoidal, (3) is equivalent to preservation of all special pullbacks, just as (2) is equivalent to preservation of just the connected pullbacks. Now any connected pullback (as in the northwest half of the following diagram) can be composed in a canonical way with a special pullback (the southeast half of the diagram) to form 
  a special connected pullback:
  $$
  
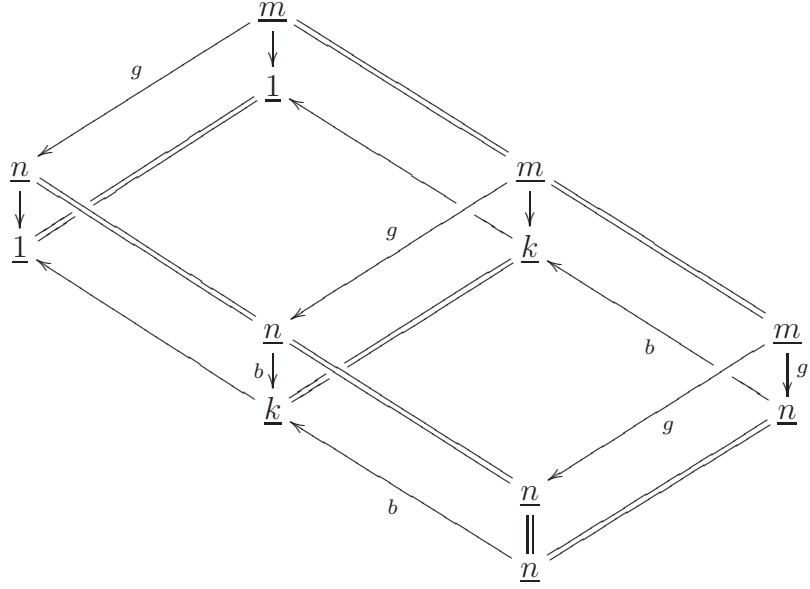
\xymatrix@C3pc@R1.2pc{
  && \ar[lldd]_{g} \un m \ar[d] \ar@{=}[rrdd] && &&\\
  && \un 1 && &&\\
  \un n \ar[d] \ar@{=}[rrdd] &&&& \un m \ar[lldd]_{g} \ar[d]  \ar@{=}[rrdd]&&\\
  \un 1 \ar@{=}[rruu]|(0.25)\hole &&&& \un k \ar[lluu]|(0.25)\hole  &&\\
  && \un n \ar[d]_b \ar@{=}[rrdd] && && \un m \ar[d]^{g} \ar[lldd]^{g}\\
  && \ar[lluu] \un k \ar@{=}[rruu]|(0.25)\hole && && \un n \ar@{=}[lldd] \ar[lluu]^b|(0.25)\hole\\
  && && \un n \ar@{=}[d] && \\
  && && \un n \ar[lluu]^b 
  }$$
  Hence, by the pullback Lemma~\ref{pbk}, if special pullbacks are preserved
  then so are connected pullbacks.
  Note that 
  $(1)\Leftrightarrow(3)$ can also be proved directly, without reference 
  to $\DD$.
\end{proof}
\begin{blanko}{Example.}
  If $g$ is the bottom degeneracy map $\un 3
  \to \un 2$ in $\unDelta$, corresponding to the active map $d^1: [2] \genmap
  [3]$ in $\Deltaact$, the special connected pullback square in Example~\ref{ex:DD-pbk} is sent to
  $$\xymatrix@C+6ex@R+0ex{
  X_3 \ar[r]^-{(d_3,d_0 d_0)} \ar[d]_{d_1} & X_2 \times X_1 \ar[d]^{d_1\times \id} \\
  X_2 \ar[r]_-{(d_2,d_0)} & X_1 \times X_1 ,
  }$$
  as in item (3) of the proposition.
  This is precisely square $\xymatrix{*+[o][F-]{1}}$ of the
  basic coassociativity argument in \ref{coasscircle1}.
\end{blanko}

\section{Proof of strong homotopy coassociativity of the incidence coalgebra}

\label{sec:proofofcoass}

We proceed to establish that, if $X$ is a decomposition space, then the
comultiplication and counit defined in \ref{comult} make $\Grpd_{/X_1}$
a coassociative and counital coalgebra in a strong homotopy sense. 

We have more
generally, for any $n\geq 0$, the generalised comultiplication maps
\begin{eqnarray}\label{Deltan}
\Delta_n:
\Grpd_{/X_1} & \longrightarrow & \Grpd_{/X_1\times\dots\times X_1}
\end{eqnarray}
defined by the spans
\begin{align}\label{Deltanspan}
X_1 \leftarrow X_n \to X_1 \times\dots\times X_1.
\end{align}
The case $n=0$ is the counit map, $n=1$ gives the identity, and $n=2$ is the
comultiplication $\Delta$ we considered above.  The coassociativity result will
follow from Lemma~\ref{lem:reasonable} and Proposition~\ref{prop:reasonable}
that say that all combinations (composites and tensor products) of these
generalised comultiplication maps are canonically equivalent whenever they have
the same source and target.  For this we exploit the category $\DD$ introduced
in \S\ref{sec:DD}, designed exactly to encode also products of the various
spaces $X_k$.

\begin{blanko}{Reasonable spans and reasonable linear functors.}
  A {\em reasonable span} in $\DD$ is a span $a \stackrel g \leftarrow m 
  \stackrel f\to b$ in 
  which $g$ is ordinalic and $f$ is segalic.  Clearly the external sum
  of two reasonable spans is reasonable, and the composite of two 
  reasonable spans is reasonable (by Lemma~\ref{lemma:DD-pbk}).
  
  Let $X:\simplexcategory\op\to\Grpd$ be a fixed decomposition space,
  and interpret it also as a monoidal functor $\ov X:\DD\to\Grpd$, 
  via Proposition~\ref{DD-universal}.
  A span in $\Grpd$ of the form
  $$
  \ov X_{a} 	\leftarrow \ov X_{m} \to \ov X_{b}
  $$
  is called \emph{reasonable} if it is induced by a reasonable span in $\DD$.
  
  Recall from \ref{bla:LIN} that a functor
  between slices of $\Grpd$ is linear if it is defined by a span {in $\Grpd$}.  
  A linear functor is called \emph{reasonable} if the defining span is reasonable.  
That is, a reasonable linear functor is a functor that is defined by a pullback along an ordinalic map followed by a lowershriek
  along a segalic map.  
\end{blanko}

\begin{lemma}\label{lem:reasonable}
	For $X$ a decomposition space,
  tensor products and composites of reasonable linear functors are 
  again reasonable linear functors.
\end{lemma}
\begin{proof}
  Cartesian products of reasonable spans in $\Grpd$ are again
  reasonable since $\ov X$ is monoidal.  Hence tensor products of
  reasonable linear functors are again reasonable. A composite of 
  reasonable linear functors is induced by
  the composite reasonable span in $\DD$, using
  Proposition~\ref{prop:DDecomp}.  Hence reasonable linear
  functors are closed under composition.
\end{proof}

The interest in these notions is of course that the generalised
comultiplication maps $\Delta_n$ of (\ref{Deltan}) are
reasonable linear functors. In fact they are the `only' reasonable linear functors:
\begin{prop}\label{prop:reasonable}
  Any reasonable linear functor $$\Grpd_{/X_1} \longrightarrow \Grpd_{/X_1\times\dots\times
  X_1}, \quad n\geq0$$ is canonically equivalent to the $n$th 
  comultiplication map $\Delta_n$.
\end{prop}

\begin{proof}
    We have to show that the only reasonable span of the form
    $X_1 \leftarrow  \prod X_{m_i}\to X_1 \times\dots\times X_1$
    is \eqref{Deltanspan}.
Indeed, the left leg must come from an ordinalic map, so since 
$X_{1}$ has only one factor, the middle object has also only one 
factor, i.e.~is the image of $\un m \to \un 1$.  On the other hand,
the right leg must be segalic, which forces $m=n$.
\end{proof}

Thus we have:

\begin{theorem}\label{thm:comultcoass}
  For $X$ a decomposition space, the slice $\infty$-category $\Grpd_{/X_1}$ has
  the structure of a strong homotopy comonoid in the symmetric monoidal
  $\infty$-category $\LIN$, with the comultiplication defined by the span
  $$
  X_1 \stackrel{d_1}\longleftarrow X_2 \xrightarrow{(d_2,d_0)} X_1 \times X_1 . 
  $$
\end{theorem}

\section{Functoriality of the incidence coalgebra  construction}

\label{sec:functorialities}

We have associated to any decomposition space $X$ its incidence
coalgebra with underlying slice $\infty$-category $\Grpd_{/X_1}$.  We now
investigate the functoriality of this construction.  Given a simplicial map
$F:X\to Y$ between decomposition spaces there are induced linear functors
$$
F\lowershriek:\Grpd_{/X_1}\to \Grpd_{/Y_1},\qquad
F\upperstar:\Grpd_{/Y_1}\to \Grpd_{/X_1},
$$
defined by postcomposition with, and pullback along, $F_1:X_1\to Y_1$.  In
this section we will give conditions on the simplicial map $F$ for these
linear functors to be coalgebra homomorphisms.

\begin{blanko}{Covariant functoriality.}    
  It is an important feature of
  \culf maps that they induce coalgebra homomorphisms:
\end{blanko}

\begin{lemma}\label{lem:coalg-homo}
  If $F:X \to Y$ is a \culf map  between decomposition spaces
  then $F\lowershriek : \Grpd_{/X_1} \to \Grpd_{/Y_1}$ is a coalgebra 
  homomorphism.
\end{lemma}

\begin{proof}
In the diagram
$$
\xymatrix{
   X_1  \ar[d]_{F_1}& \ar[l]_{g} \dlpullback X_n\ar[r]^-{f} \ar[d]^{F_n}
   &
   X_1^n \ar[d]^{F_1^n}
   \\
   Y_1  & \ar[l]^{g'} Y_n\ar[r]_-{f'} & Y_1^n
   }
   $$
   the left-hand square is a pullback since $F$ is conservative (case $n=0$) and ULF 
   (cases $n>1$).  Hence by the Beck--Chevalley condition we have an equivalence 
   of functors $g'{}\upperstar \circ F_1{}\lowershriek \simeq
   F_n{}\lowershriek\circ g{}\upperstar$, and by postcomposing with 
   $f'\lowershriek$ we arrive at the coalgebra homomorphism condition
$
\Delta'_nF_1{}\lowershriek
\cong
F_1{}\lowershriek^n\Delta_n
$. 
\end{proof}

\begin{blanko}{Remark.}
  If $Y$ is a Segal space, then the statement can be improved to an if-and-only-if
  statement.
\end{blanko}

\begin{blanko}{Example.}
  An important class of \culf maps are the dec maps 
  (Proposition~\ref{Dec=Segal+cULF}):
  $$
  d_\bot : \Decbot{X} \to X \qquad \text{ and } \qquad d_\top : 
  \Dectop{X} \to X .
  $$
  Many coalgebra maps in the classical theory of incidence
  coalgebras, notably reduction maps, are induced from decalage in this way,
  as we shall see in Section~\ref{sec:ex}, and as further
  amplified in \cite{GKT:ex}.
\end{blanko}

\begin{blanko}{Contravariant functoriality.}
  There is also a contravariant
  functoriality for certain simplicial maps, which we briefly explain,
  although it will not be needed elsewhere in this paper.
  
  We will say that a functor between decomposition spaces $F: X \to Y$ is {\em relatively
  Segal} when for any `spine' (i.e.~an inclusion of the string of principal edges
  into a simplex)
  $$\textstyle
  \Delta[1] \underset{\Delta[0]}\coprod \dots \underset{\Delta[0]}\coprod \Delta[1]
  \longrightarrow
  \Delta[n]  ,
  $$
  the space of fillers in the diagram
  $$\xymatrix{
    \Delta[1] \underset{\Delta[0]}\coprod \dots \underset{\Delta[0]}\coprod \Delta[1]
    \ar[r]\ar[d] & X \ar[d] \\
     \Delta[n] \ar[r]\ar@{-->}[ru] & Y
  }$$
  is contractible.
  Note that the precise condition is that the following square is a pullback:
  $$\xymatrix{
     \Map(\Delta[n], X)\drpullback \ar[r]\ar[d] & \Map(\Delta[1] \underset{\Delta[0]}\coprod \dots \underset{\Delta[0]}\coprod \Delta[1],X) \ar[d] \\
     \Map(\Delta[n],Y) \ar[r] & \Map(\Delta[1] \underset{\Delta[0]}\coprod \dots \underset{\Delta[0]}\coprod \Delta[1],Y).
  }$$
  This can be rewritten
  \begin{equation}\label{eq:relSegal}
  \vcenter{\xymatrix{
     X_n\drpullback \ar[r]\ar[d] & X_1 \times_{X_0} \cdots \times_{X_0} X_1 \ar[d] \\
     Y_n \ar[r] & Y_1 \times_{Y_0} \cdots \times_{Y_0} Y_1 .
  }}    
  \end{equation}
  (Hence the ordinary Segal condition for a simplicial space $X$ is the case
  where $Y$ is a point.)
\end{blanko}

\begin{prop}
  If $F:X \to Y$ is relatively Segal and $F_0: X_0 \to Y_0$
  is an equivalence, then
  $$
  F\upperstar : \Grpd_{/Y_1} \to \Grpd_{/X_1}
  $$
  is naturally a coalgebra homomorphism, that is, there is a canonical equivalence of functors
$$
  \Delta_n F_1{}\upperstar 
  \simeq
  F_1{}\upperstar {}^n\Delta'_n .
  $$
where $\Delta_n$ and $\Delta_n'$ are the comultiplication maps for $X$ and $Y$.
\end{prop}

\begin{proof}
  In the diagram
  $$
  \xymatrix{
  X_1  \ar[d]_{F_1}& \ar[l]_{g}  X_n\ar[r]^-{f} \ar[d]_{F_n}
  &
  X_1^n \ar[d]^{F_1^n}
  \\
  Y_1  & \ar[l]^{g'} Y_n\ar[r]_-{f'}
  & Y_1^n
  }
  $$
  we claim that the right-hand square
  is a pullback for all $n$. In this case,
  by the Beck--Chevalley condition, we would have an equivalence 
  of functors $f\lowershriek \circ F_n{}\upperstar \simeq
  F_1^n{}\upperstar \circ f'{}\lowershriek$, and by precomposing with 
  $g'{}\upperstar $ we would arrive at the required coalgebra homomorphism condition.
  
    The claim for $n=0$ amounts to 
    $$
  \xymatrix{
   \drpullback X_0\ar[r]^-{f} \ar[d]_{F_0}
  &
  1 \ar[d]
  \\
   Y_0\ar[r]_-{f'} & 1
  }
  $$
  which is precisely to say that $F_0$ is an equivalence.
  For $n>1$ we can factor the square as
      $$
  \xymatrix{
   \drpullback X_n\ar[r]^-{f} \ar[d]_{F_n}
   & X_1\times_{X_0} \cdots \times_{X_0} X_1 \ar[d]^{F_1^n} \ar[r]
   & X_1 \times \cdots \times X_1 \ar[d]^{F_1^n}
   \\
   Y_n\ar[r]_-{f'} 
   & Y_1\times_{Y_0} \cdots \times_{Y_0} Y_1 \ar[r]
   & Y_1 \times \cdots \times Y_1
  }
  $$
  Here the left-hand square is a pullback since $F$ is relatively Segal.
  It remains to prove that the right-hand square is a pullback.
  For the case $n=2$, this whole square is the pullback of the square
  $$\xymatrix{
     X_0\drpullback \ar[r]\ar[d] & X_0 \times X_0 \ar[d] \\
     Y_0 \ar[r] & Y_0 \times Y_0
  }$$
  which is a pullback precisely when $F_0$ is mono.  But we have
  assumed it is even an equivalence.
  The general case $n>2$ is easily obtained from the $n\!=\!2$ case by an iterative
  argument.
\end{proof}

\begin{blanko}{Remarks.}
  It should be mentioned that in order for contravariant functoriality to
  preserve finiteness as in \cite{GKT:DSIAMI-2}, and hence restrict to
  coefficients in homotopy-finite $\infty$-groupoids, it is necessary
  furthermore to require that $F$ is finite, cf.~\cite{GKT:HLA}.
  
  When both $X$ and $Y$ are Segal spaces, then the relative Segal condition is
  automatically satisfied, because the horizontal maps in \eqref{eq:relSegal}
  are then equivalences.  In this case, we recover the classical results on
  contravariant functoriality by
  Content--Lemay--Leroux~\cite[Prop.~5.6]{Content-Lemay-Leroux} and
  Leinster~\cite{Leinster:1201.0413}, where the only condition is that the
  functor be bijective on objects (in addition to requiring $F$ finite,
  necessary since they work on the level of vector spaces).
\end{blanko}

\section{Monoidal decomposition spaces}
\label{sec:monoids-a}

The $\infty$-category of decomposition spaces (as a full subcategory of simplicial
spaces), has finite products.  Hence there is a symmetric monoidal structure
on the $\infty$-category $\Decomp^{\culfsymb}$ of decomposition spaces and \culf maps.
We still denote this product as $\times$, although of course it is not the
categorical product in $\Decomp^{\culfsymb}$.

\begin{blanko}{Definition.}
  A {\em monoidal decomposition space} is a monoid object $(X,m,e)$ in
  $(\Decomp^{\culfsymb}, \times, 1)$. 
  A {\em monoidal functor} between monoidal decomposition spaces is a monoid
  homomorphism in $(\Decomp^{\culfsymb}, \times, 1)$.
\end{blanko}

  By this we mean a monoid in the homotopy sense, that is, 
  an algebra in the sense of
  Lurie~\cite[\S4.1]{Lurie:HA}.  We do not wish at this point
  to go into the technicalities of this notion, since
  in our examples, the algebra structure will be given
  simply by sums (or products).

\begin{blanko}{Example.}\label{extensive}
  Recall that a category $\EE$ with finite sums is {\em extensive}
  \cite{Carboni-Lack-Walters} when the 
  natural functor $\EE_{/A} \times \EE_{/B} \to \EE_{/A+B}$ is an equivalence.
  The fat nerve of an extensive $1$-category is a monoidal decomposition space.
  The multiplication is given by taking sum, the neutral object by the initial 
  object, and the extensive property ensures precisely that, given a factorisation of
  a sum of maps, each of the maps splits into a sum of maps in a unique way.

  A key example is the category of sets, or of finite sets.  Certain
  subcategories, such as the category of finite sets and surjections, or the
  category of finite sets and bijections, inherit the crucial property $\EE_{/A}
  \times \EE_{/B} \simeq \EE_{/A+B}$.  They fail, however, to be extensive in
  the strict sense, since the monoidal structure $+$ in these cases is not the
  categorical sum.  Instead they are examples of {\em monoidal extensive}
  categories, meaning a monoidal category $(\EE, \boxplus, 0)$ for which
  $\EE_{/A} \times \EE_{/B} \to \EE_{/A \boxplus B}$ is an equivalence (and it
  should then be required separately that also $\EE_{/0}\simeq 1$).  The fat
  nerve of a monoidal extensive $1$-category is a monoidal decomposition space.
\end{blanko}

\begin{lemma}\label{dec-mon-is-mon}
  The lower dec of a monoidal decomposition space has again a natural monoidal 
  structure, and the dec map is a monoidal functor. The same is true for the upper dec.
\end{lemma}

\begin{blanko}{Bialgebras.}\label{sec:monoids}
  For a monoidal decomposition space the resulting coalgebra is also a
  bialgebra.  Indeed, the fact that the monoid multiplication is \culf
  means that it induces a coalgebra homomorphism, and similarly with the
  unit.  Note that in the bialgebras arising like this, the algebra
  and coalgebra are not on entirely equal footing: while the
  comultiplication is induced from internal, simplicial data in $X$, the
  multiplication is induced by extra structure (the monoidal structure),
  and is given by spans with trivial pullback component.  In examples
  coming from combinatorics, the monoid structure will typically be given
  by categorical sum.
\end{blanko}

\begin{blanko}{Running example: the Hopf algebra of rooted trees.}\label{ex:trees-bialg}
  The decomposition space $\dstrees$ of Examples~\ref{ex:trees-decomp} and
  \ref{ex:trees-coalg} has a canonical monoidal structure given by disjoint
  union.  Recall that $\dstrees_k$ is the groupoid of forests with $k-1$ compatible 
  admissible cuts.
  The disjoint
  union of two such structures is given by taking the disjoint union of the
  underlying forests, with the cuts concatenated.
  This clearly defines a simplicial map from
  $\dstrees\times \dstrees$ to $\dstrees$.  To say that it is \culf is to establish that squares
  like this are pullbacks:
\newcommand{\twoforests}{
  \bsegment
	\htext (-47 15){\large $\big($}
	\htext (48 15){\large $\big)$}
	\htext (0 3){$,$}
	\move (-19 0)
	\bsegment
	  \move (0 0) \Onedot \lvec (-7 15) \Onedot 
	  \lvec (-12 30) \Onedot \move (-7 15) \lvec (-2 30) \Onedot
	  \move (0 0) \lvec (7 15) \Onedot \lvec (8 30) \Onedot
	\esegment
	\move (14 0)
	\bsegment
	  \move (0 0) \Onedot \lvec (0 15) \Onedot 
	  \move (13 0) \Onedot \lvec (8 15) \Onedot \lvec (7 30) \Onedot
	  \move (13 0) \lvec (20 15) \Onedot
	\esegment
  \esegment
}
\newcommand{\joinedforests}{
  \bsegment
	\move (-14 0)
	\bsegment
	  \move (0 0) \Onedot \lvec (-7 15) \Onedot 
	  \lvec (-12 30) \Onedot \move (-7 15) \lvec (-2 30) \Onedot
	  \move (0 0) \lvec (7 15) \Onedot \lvec (8 30) \Onedot
	\esegment
	\move (12 0)
	\bsegment
	  \move (-2 0) \Onedot \lvec (-2 15) \Onedot 
	  \move (13 0) \Onedot \lvec (8 15) \Onedot \lvec (7 30) \Onedot
	  \move (13 0) \lvec (20 15) \Onedot
	\esegment
  \esegment
}
\medskip\par\noindent\hspace*{2cm}
$
\vcenter{\xymatrix{
\dstrees_1\times \dstrees_1 \ar[d]_{+} & \ar[l]_{d_1} \dlpullback \dstrees_2 \times \dstrees_2 \ar[d]^{+} \\
\dstrees_1 & \ar[l]^{d_1} \dstrees_2}}
\quad\qquad
\vcenter{\begin{texdraw}
      \setunitscale 0.5
      \footnotesize    
      \move (0 0)
	  \joinedforests
      \move (200 0)
	  \joinedforests
	  \move (160 22)
      \bsegment
		\move (5 -2) \clvec (60 15)(20 -27)(80 -14)
	  \esegment
      \move (0 120)
	  \twoforests
      \move (200 120)
	  \twoforests
      \move (200 120)
	  \bsegment
		\move (-35 20)
		\bsegment
		\move (0 0) \clvec (10 4)(20 4)(30 0)
		\esegment
		\move (9 5)
		\bsegment
		\move (0 0) \clvec (10 4)(20 4)(30 0)
		\esegment
	  \esegment
      \move (100 135)
      \bsegment
      \move (25 -3) \rlvec (0 6)
      \move (25 0) \ravec (-50 0) 
      \esegment
      \move (100 15)
      \bsegment
      \move (25 -3) \rlvec (0 6)
      \move (25 0) \ravec (-50 0) 
      \esegment
      \move (0 60)
      \bsegment
      \move (-3 33) \rlvec (6 0)
      \move (0 33) \avec (0 -3) 
      \esegment
      \move (200 60)
      \bsegment
      \move (-3 33) \rlvec (6 0)
      \move (0 33) \avec (0 -3) 
\esegment
    \end{texdraw}}
  $
\medskip\par\noindent  
  But this is clear: a pair of forests each with an admissible cut
  can be uniquely reconstructed if we know what the two underlying forests are
  (an element in $\dstrees_1\times \dstrees_1$) and we know how the disjoint union
  is cut (an element in $\dstrees_2$) --- provided of course that we can
  identify the disjoint union of those two underlying forests with the 
  underlying forest of the disjoint union (which is to say that the data
  agree down in $\dstrees_1$).
  It follows that the resulting incidence coalgebra is also a bialgebra, 
  in fact Hopf algebra.
\end{blanko}

\begin{prop}\label{bialg-hm}
  If $F:X\to Y$ is a monoidal \culf functor between monoidal decomposition
  spaces, then $F\lowershriek : \Grpd_{/X_1} \to \Grpd_{/Y_1}$ is a
  bialgebra homomorphism.
\end{prop}

\section{Examples}

\label{sec:ex}

\begin{blanko}{Injections and the monoidal groupoid of sets under sum.}\label{I=DecB}
  Let $\mathbf{I}$ be the fat nerve of the category of finite sets and
  injections, and let $\mathbf{B}$ be the monoidal nerve of the monoidal
  groupoid $(\B, +, 0)$ of finite sets and bijections (see \ref{prop:BM}).  If
  we construct the incidence coalgebra of the decomposition space $\mathbf I$
  and impose the equivalence relation `having isomorphic complements' then, as
  observed by D\"ur~\cite{Dur:1986}, we obtain the binomial coalgebra.  The
  binomial coalgebra also arises directly as the incidence coalgebra of
  $\mathbf{B}$, and D\"ur's reduction arises as a \culf functor from a decalage:
  
\begin{lemma}\label{lem:DecB}
There is a levelwise equivalence of simplicial groupoids 
  \begin{align*}
   \Decbot{\mathbf B}  & \stackrel\simeq\longrightarrow  \mathbf{I}  \\
   \intertext{given in degree $k$ by}
    (x_0,\dots,x_k) & \longmapsto  [x_0\subseteq x_0+x_1 \subseteq  \dots \subseteq x_0+\dots+x_k],\\
    (y_0,y_1\setminus y_0,\dots,y_k\setminus y_{k-1})
                          &  \;\longmapsfrom  [y_0\subseteq y_1 \subseteq  \dots \subseteq y_k].
          \end{align*}
    The canonical \culf functor 
  $$
  d_\bot: \Decbot{\mathbf B} \to \mathbf B,\qquad (x_0,\dots,x_k)\mapsto (x_1,\dots,x_k)$$
  defines the reduction map $r:\mathbf{I}\to \mathbf B$.
\end{lemma}

The equivalence may also be represented using diagrams reminiscent of those
in Waldhausen's \Sdot-construction, cf.~\ref{Sdot}
below. As an example, both groupoids $\mathbf{I}_3 $ and
$(\Decbot{\mathbf B})_3=\mathbf B_4$ are equivalent to the groupoid of
diagrams
  $$\xymatrix@R-0.8pc{
   &   &   &x_3 \ar[d] 
\\
   &   &x_2 \ar[d] \ar[r] & x_2+x_3 \ar[d] 
\\
   &x_1\ar[r]\ar[d] & x_1+x_2 \ar[d] \ar[r] & x_1+x_2+x_3 \ar[d] 
\\
x_0 \ar[r] & x_0+x_1 \ar[r] & x_0+x_1+x_2\ar[r] & x_0+x_1+x_2+x_3
  }$$
For $i<3$, the face maps $d_i:\mathbf{I}_3\to \mathbf{I}_2$ and
$d_{i}:(\Decbot{\mathbf B})_3\to (\Decbot{\mathbf B})_2$ act by erasing the
column beginning $x_i$ and the row beginning $x_{i+1}$. The top face map
$d_3$ erases the last column.
The face map $d_0:
\mathbf{B}_4\to\mathbf{B}_3$ erases the bottom row.

Both $\mathbf{I}$ and $\mathbf{B}$ are monoidal decomposition spaces under
disjoint union, and $\mathbf{I}\simeq \Decbot{\mathbf B} \to \mathbf{B}$
is a monoidal functor by Lemma~\ref{dec-mon-is-mon}, inducing a
homomorphism of bialgebras $\Grpd_{/{\mathbf{I}}_1}\to\Grpd_{/\mathbf B_1}$
by Proposition \ref{bialg-hm}, which is the reduction map described by
D\"ur~\cite{Dur:1986}.

Recall from \cite[2.3]{GKT:HLA} that 
the functors $\name S : 1\to\mathbf B_1$, $1\mapsto S$, play the role of a basis of
$\Grpd_{/\mathbf B_1}$ as $S$ ranges over $\pi_0\mathbf B_1$. 
The comultiplication on $\Grpd_{/\mathbf B_1}$ is 
  \begin{align*}
    \Delta&(\name S) 
\;= 
\sum_{A+B=S}
 \name A \tensor \name B 
  \end{align*}
  (where the sum is more specifically over all $A ,B\!\subset\!  S,\; A\cup B
  \!=\!  S,\; A\cap B \!=\! \varnothing$).  The decomposition space {\bf B} is
  {\em locally finite} (see \cite[\S 7]{GKT:DSIAMI-2}), that is,
  $\mathbf B_1$ has finite automorphism groups and the maps
      $s_0:\mathbf B_0\to \mathbf B_1$ and 
$d_1:\mathbf B_2\to\mathbf B_1$ are finite. Therefore we can take cardinality (as in
  \cite{GKT:HLA}), giving the classical binomial coalgebra spanned by symbols
  $\delta_n$ (the cardinality of $\name n : 1 \to {\bf B}_1$) with
  $$
  \Delta(\delta_n) = \sum\limits_{a+b=n} \frac{n!}{a!\,b!}\,\delta_a \tensor 
  \delta_b.
  $$ 
  As a bialgebra we have $(\delta_1)^n=\delta_n$ and one recovers the
  comultiplication from $\Delta(\delta_n)= \big( \delta_0 \tensor \delta_1 +
  \delta_1 \tensor \delta_0 \big)^n$.

  The objective level is much richer.
  The linear dual~\cite{GKT:DSIAMI-2} of $\Grpd_{/\mathbf B_1}$ is
  $\Grpd^{\mathbf B_1}$, the
  category of groupoid-valued species~\cite{Baez-Dolan:finset-feynman}, 
  \cite{Kock:MFPS28}, and its multiplication is the  monoidal structure
  given by the convolution 
  formula
  $$
  (F*G)[S] = \sum_{A+B=S} F[A] \times G[B] ,
  $$
  which is precisely the Cauchy product of species (see \cite{Aguiar-Mahajan}).
  The cardinality of this monoidal category is
  the profinite-dimensional vector space $\Q^{\pi_0\mathbf B_1}$ with pro-basis
  given by the symbols $\delta^n$ (dual to $\delta_n$), with convolution product
  $$
  \delta^a* \delta^b = \frac{n!}{a!\,b!}\,\delta^{a+b} .
  $$
  This is isomorphic to the algebra $\Q[[z]]$, where $\delta^n$  
  corresponds to  $z^n/n!$ and the cardinality of a species $F$ corresponds 
  precisely
  to its exponential generating series \cite{JoyalMR633783}.
\end{blanko}

\begin{blanko}{Graphs.}\label{ex:graphs}
  The following coalgebra of graphs is due to
  Schmitt~\cite[\S 12]{Schmitt:1994}.
  For a graph $G$ with vertex set $V$ (admitting multiple edges and loops),
  and a subset $U \subset V$, define
  $G|U$ to be the graph whose vertex set is $U$, and whose edges are those
  edges of $G$ both of whose incident vertices belong to $U$.
  On the vector space spanned by iso-classes of graphs,
  define a comultiplication by the rule
  $$
  \Delta(G) = \sum_{A+B=V} G|A \tensor G|B .
  $$
  
  This coalgebra is the cardinality of the coalgebra of a decomposition space
  (cf.~\cite{GKT:ex}), but not directly of a category.  Indeed, define a
  simplicial groupoid with $\mathbf G_1$ the groupoid of graphs, and more
  generally let $\mathbf G_k$ be the groupoid of graphs with an ordered
  partition of the vertex set into $k$ (possibly empty) parts.  In particular,
  $\mathbf G_0=\terminal$ is the contractible groupoid consisting only of the
  empty graph.  The outer face maps delete the first or last part of the graph,
  and the inner face maps join adjacent parts.  The degeneracy maps insert an
  empty part.  It is clear that this is not a Segal space: a graph structure on
  a given set cannot be reconstructed from knowledge of the graph structure of
  the parts of the set, since chopping up the graph and restricting to the parts
  throws away all information about edges going from one part to another.  One
  can easily check that it is a decomposition space.  It is clear that the
  cardinality of the resulting coalgebra is Schmitt's coalgebra of
  graphs.  Note that disjoint union of graphs makes this into a bialgebra, in
  fact a Hopf algebra.
\end{blanko}

\begin{blanko}{Running example: the Hopf algebra of rooted trees.}\label{ex:CK}
  D\"ur~\cite[IV \S3]{Dur:1986} 
  gave an incidence-coalgebra construction of
  the Butcher--Connes--Kreimer Hopf algebra by starting with the category $\CC$ of
  rooted forests and root-preserving inclusions, generating a coalgebra (in our
  language the incidence coalgebra of the fat nerve of $\CC$), and imposing the
  equivalence relation that identifies two root-preserving forest inclusions if
  their complement crowns are isomorphic forests.  To be precise, this yields
  the opposite of the Butcher--Connes--Kreimer coalgebra, in the sense that the
  factors $P_c$ and $R_c$ are interchanged.  To remedy this, one should use
  $\CC\op$ instead of $\CC$.
  
  As we have seen (in our running example~\ref{ex:preintro}, 
  \ref{ex:trees-decomp}, \ref{ex:trees-coalg}, \ref{ex:trees-bialg}),
  we can obtain the Butcher--Connes--Kreimer Hopf algebra directly from the
  (monoidal)
  decomposition space (see~\cite{GKT:restriction} for more details) $\dstrees$
  where $\dstrees_1$ denotes the groupoid
  of forests, and $\dstrees_2$ is the groupoid of forests with an
  admissible cut, and so on. 
  The relationship with D\"ur's construction is this 
  (cf.~\cite{GKT:restriction}): the
  `raw' decomposition space $\fatnerve(\CC\op)$ is the decalage 
  of $\dstrees$,
  in close analogy with Lemma~\ref{lem:DecB}:
  $$
  \Dectop{\dstrees} \simeq \fatnerve(\CC\op) .
  $$
  Furthermore, under this identification, the dec map $\Dectop{\dstrees}
  \to \dstrees$, always a (monoidal) \culf functor, realises
  precisely D\"ur's reduction: on $(\fatnerve\CC\op)_1
  \to \dstrees_1$ it sends a root-preserving forest inclusion to its crown,
  that is, its complement.  More generally, on $(\fatnerve\CC\op)_k \to \dstrees_k$ it
  sends a sequence
  of forest inclusions $ F_0 \subset F_1 \subset \dots \subset F_k $ to
  $$
  F_1\shortsetminus F_0 \subset \dots \subset F_k \shortsetminus F_0 .
  $$ 
\end{blanko}

\begin{blanko}{Restriction species, directed restriction species, and operads.}
  The graph example $\mathbf G$ is an example of a decomposition space coming
  from a restriction species in the sense of Schmitt~\cite{Schmitt:hacs} (see
  also \cite{Aguiar-Mahajan}).  The tree example $\dstrees$ is an example of a
  decomposition space coming from a {\em directed restriction species}, a notion
  introduced in~\cite{GKT:restriction}, formalising the idea of considering only
  decompositions compatible with an underlying poset structure, as exemplified
  clearly by the notion of admissible cut.

  While the decomposition space $\dstrees$ is not a Segal space, it
  admits important variations which {\em are} Segal spaces, namely by replacing
  the combinatorial trees above by various kinds of {\em operadic} trees.  These
  yield only bialgebras instead of Hopf algebras, but the Segal property has
  been exploited to good effect in various contexts \cite{Kock:1109.5785},
  \cite{GalvezCarrillo-Kock-Tonks:1207.6404}, \cite{Kock:1411.3098},
  \cite{Kock:1512.03027}.  These examples are subsumed in the general notion of
  incidence bialgebra of an operad, cf.~\cite{GKT:ex}
  and \cite{Kock-Weber:FdB}.
\end{blanko}

\begin{blanko}{$q$-binomials: $\F_q$-vector spaces.}\label{ex:q}
  Consider the finite field $\mathbb F_q$ with $q$ elements.  The $q$-binomial
  coalgebra (see D\"ur~\cite[1.54]{Dur:1986}) may be obtained as a
  reduction of the incidence coalgebra of the category $\vect$, of
  finite-dimensional $\mathbb F_q$-vector spaces and $\mathbb F_q$-linear
  injections, by identifying two injections if their cokernels are isomorphic.
  
  The same coalgebra can be obtained without reduction as follows.  Put $\mathbf
  V_0 = \terminal$ (the contractible groupoid of $0$-dimensional 
  vector spaces), let $\mathbf V_1$ be the maximal subgroupoid of $\vect$, and let
  $\mathbf V_2$ be the groupoid of short exact sequences.  The span
  $$
   \xymatrixrowsep{4pt}
  \xymatrix {
  \mathbf V_1 & \ar[l] \mathbf V_2 \ar[r]  &\mathbf V_1\times\mathbf V_1 \\
  E & \ar@{|->}[l] [ E' \!\to\! E \!\to\! E''] \ar@{|->}[r]  & (E',E'')
  }$$
  (together with the span $\mathbf V _1 \leftarrow \mathbf V_0 \to 1$) defines a
  coalgebra on $\Grpd_{/\mathbf V_1}$ which (after taking cardinality) is the
  $q$-binomial coalgebra, without further reduction.  The groupoids and maps
  involved are part of a simplicial groupoid $\mathbf V: \simplexcategory\op\to\Grpd$,
  namely the Waldhausen \Sdot-construction of $\vect$, which
  is a decomposition space but not a Segal space (cf.~\ref{Sdot} below).  The lower dec
  of $\mathbf V$ is naturally equivalent
  to the fat nerve of $\vect$, and the
  comparison map $d_0$ is the reduction map of D\"ur.

  Although we have postponed the notion of the dual incidence {\em algebra}
  to \cite{GKT:DSIAMI-2}, we wish to mention that in this case the
  incidence algebra is $\Grpd^{\mathbf V_1}$,
  which is the category of groupoid-valued $q$-species, and the 
  convolution tensor product resulting from our constructions
  is the {\em external product of $q$-species}
  of Joyal--Street~\cite{Joyal-Street:GLn} (except that they work
  with vector-space valued $q$-species).  A main contribution of 
  \cite{Joyal-Street:GLn} is to show that this monoidal structure carries
  a non-trivial braiding.  This is a very interesting structure, which
  cannot be seen after taking cardinality.
  
  One can compute explicitly (see \cite{GKT:ex}) the section coefficients
  of the comultiplication (or the convolution product) to find the Hall numbers
   $$
  \frac{\norm{\text{SES}_{k,n,n-k}}}{\norm{\Aut(\F_q^k)}\norm{\Aut(\F_q^{n-k})}}  
  =  { n \choose k}_q ,
  $$
  where $\text{SES}_{k,n,n-k}$ denotes the groupoid of
short exact sequences of fixed vector spaces of
  dimensions $k$, $n$, and $n-k$.

  This example is a special case of the following general construction
  with wide-ranging ramifications and consequences.
\end{blanko}

\begin{blanko}{Waldhausen \Sdot-construction of an abelian category~\cite{Waldhausen}.}
  \label{Sdot}
  We follow Lurie~\cite[Subsection 1.2.2]{Lurie:HA} for the account of
  Waldhausen's \Sdot-construction.  For $I$ a linearly ordered set,
let $\Ar(I)$ denote the category of arrows in $I$: the objects
are pairs of elements $i\leq j$ in $I$, and the morphisms are relations
$(i,j) \leq (i',j')$ whenever $i\leq i'$ and $j\leq j'$.
A {\em gap complex} in an abelian category $\mathcal A$ is a functor
$F: \Ar(I) \to \mathcal A$ such  that

\begin{enumerate}
  \item For each $i \in I$, the object $F(i,i)$ is a zero object.

  \item For every $i \leq j \leq k$, the associated diagram
  $$\xymatrix{
 \makebox[0ex][r]{$0\simeq{}$} F(j,j) \ar@{ >->}[r] & F(j,k) \\
  F(i,j) \ar@{->>}[u] \ar@{ >->}[r] & F(i,k) \ar@{->>}[u]
  }$$
  is a pushout (or equivalently a pullback).
\end{enumerate}
Since the pullback of a monomorphism is always a monomorphism, and the pushout
of an epimorphism is always an epimorphism, it follows that automatically the
horizontal maps are monomorphisms and the vertical maps are epimorphisms, as
already indicated with the arrow typography.  Altogether, it is just a
convenient way of saying `short exact sequence' or `(co)fibration sequence'.

Let $\operatorname{Gap}(I,\mathcal A)$ denote the full subcategory of
$\Fun(\Ar(I),\mathcal A)$ consisting of the gap complexes, and
$\operatorname{Gap}(I,\mathcal A)^\eq$ its maximal subgroupoid.  The assignment
$$
[n] \mapsto \operatorname{Gap}([n],\mathcal A)^{\eq}
$$
defines a simplicial space $\Sdot\mathcal A: \simplexcategory\op\to\Grpd$, which by
definition is the Waldhausen \Sdot-construction of $\mathcal A$.
Intuitively (or essentially), the
groupoid $\operatorname{Gap}([n],\mathcal A)^{\eq}$ has as objects
staircase diagrams like the following (picturing $n=4$).
$$\xymatrix{
&&& A_{34} \\
&& A_{23} \ar@{ >->}[r] & A_{24} \ar@{->>}[u] \\
& A_{12} \ar@{ >->}[r] & A_{13} \ar@{->>}[u]\ar@{ >->}[r] & A_{14} \ar@{->>}[u] \\
A_{01} \ar@{ >->}[r] & A_{02} \ar@{->>}[u]\ar@{ >->}[r] & A_{03} 
\ar@{->>}[u]\ar@{ >->}[r] & A_{04} \ar@{->>}[u]
}$$
\end{blanko}
Informally, the face map $d_i$ `erases' all objects containing an $i$ index.
The degeneracy map $s_i$ repeats the $i$th row and the $i$th column.

A string of composable monomorphisms 
$(A_{1}\rightarrowtail A_{2}\rightarrowtail \dots \rightarrowtail A_{n})$
determines, up to canonical isomorphism,  short exact  
sequences $A_{ij}\rightarrowtail A_{ik}\twoheadrightarrow 
A_{jk}=A_{ij}/A_{ik}$ with $A_{0i}=A_i$.
Hence the whole diagram can be reconstructed up to isomorphism from the
bottom row.  (Similarly, since epimorphisms have uniquely determined kernels,
the whole diagram can also be reconstructed from the last column.)

We have 
\begin{align*}
d_0(A_1\rightarrowtail A_2\rightarrowtail \dots \rightarrowtail A_{n})&= 
(A_2/A_1\rightarrowtail \dots \rightarrowtail A_{n}/A_1)
\\s_0(A_1\rightarrowtail A_2\rightarrowtail \dots \rightarrowtail A_{n})&= 
(0\rightarrowtail 
A_1\rightarrowtail 
A_2\rightarrowtail \dots \rightarrowtail A_{n}) 
\end{align*}
The simplicial maps $d_i,s_i$ for $i\geq1$ are more straightforward: the 
simplicial set $\Decbot{(\Sdot\mathcal A)}$ is just the fat
nerve of $\mathrm{mono}(\mathcal A)$.

\begin{lemma}
  The projection $S_{n+1} \mathcal A \to \Map([n], \mathrm{mono}(\mathcal A))$
  is an equivalence. 
  Similarly the projection 
  $S_{n+1} \mathcal A \to \Map([n], \mathrm{epi}(\mathcal A))$ is an equivalence.
\end{lemma}
More precisely:

\begin{prop}
  These equivalences assemble into levelwise simplicial 
equivalences
$$
\Decbot{(\Sdot\mathcal A)} \simeq \fatnerve(\mathrm{mono}(\mathcal A))
$$
$$
\Dectop{(\Sdot\mathcal A)} \simeq \fatnerve(\mathrm{epi}(\mathcal A)) .
$$
\end{prop}

\begin{theorem}
  The Waldhausen \Sdot-construction of an abelian category $\mathcal A$
  is a decomposition space.
\end{theorem}
\begin{proof}
  For convenience we write $\Sdot\mathcal A$ simply as $\Sdot$.
  The previous proposition already implies that the two $\operatorname{Dec}$s
  of $\Sdot$
  are Segal spaces.  By Theorem~\ref{thm:decomp-dec-segal}, it is 
  therefore enough to
  establish that the squares
  $$
  \xymatrix{S_1 \ar[r]^{s_1} \ar[d]_{d_0} & S_2 \ar[d]^{d_0} \\
  S_0 \ar[r]_{s_0} & S_1
  }
  \qquad \qquad
  \xymatrix{S_1 \ar[r]^{s_0} \ar[d]_{d_1} & S_2 \ar[d]^{d_2} \\
  S_0 \ar[r]_{s_0} & S_1
  }
  $$
  are pullbacks.
  Since a zero object has no nontrivial automorphisms,  $s_0 : S_0 \to S_1$ is a monomorphism of groupoids, given by the inclusion of
the groupoid of zero objects into $S_1 = \mathcal A^{\operatorname{iso}}$.
  The map $d_0 : S_2 \to S_1$
  sends a monomorphism to its cokernel, and its fibre over a zero 
  object is the full subgroupoid
  of $S_2$ consisting of those monomorphisms whose cokernel is zero.  Clearly these
  are precisely the isos, so the fibre is just $\mathcal A^{\operatorname{iso}} = S_1$.
  The other pullback square is established similarly, but arguing with epimorphisms
  instead of monomorphisms.
\end{proof}
\begin{blanko}{Remark.}
  Waldhausen's \Sdot-construction was designed for more general categories than
  abelian categories, namely what are now called Waldhausen categories, where
  the cofibrations play the role of the monomorphisms, but where there is no stand-in
  for the epimorphisms.  The theorem does not generalise to Waldhausen categories in
  general, since in that case $\Dectop{(\Sdot)}$ is not necessarily
  a Segal space of any class of arrows.
\end{blanko}

\begin{blanko}{Waldhausen \Sdot\ of a stable $\infty$-category.}
  The same construction works in the $\infty$-setting, by considering stable
  $\infty$-categories instead of abelian categories.  Let $\mathcal A$ be a
  stable $\infty$-category (see Lurie~\cite[\S1.1.1]{Lurie:HA}).  Just as in the abelian
  case, the assignment
  $$
  [n] \mapsto \operatorname{Gap}([n],\mathcal A)^{\eq}
  $$
  defines a simplicial space $\Sdot\mathcal A: \simplexcategory\op\to\Grpd$, which by
  definition is the Waldhausen \Sdot-construction of $\mathcal A$.
  Note that in the case of a stable $\infty$-category, in contrast to the abelian 
  case, every map can arise as either horizontal or vertical arrow in a gap 
  complex.  Hence the role of monomorphisms (cofibrations)
  is played by all maps, and the role of epimorphisms is also played by all maps.
\end{blanko}

\begin{lemma}
  Suppose $\mathcal A$ is a stable $\infty$-category.
  For each $k\in \N$, the two projection functors
$S_{k+1}\mathcal A \to \Map(\Delta[k], \mathcal A)$
are equivalences.
\end{lemma}
From the description of the face and degeneracy maps, 
the following more precise result follows readily,
comparing with the fat nerves in the sense of \ref{fatinfty}:

\begin{prop}
  For $\mathcal A$ a stable $\infty$-category, we have natural
  (levelwise) simplicial equivalences
$$
\Decbot{(\Sdot\mathcal A)} \simeq \fatnerve\mathcal A
$$
$$
\Dectop{(\Sdot\mathcal A)} \simeq \fatnerve\mathcal A.
$$
\end{prop}

\begin{theorem}\label{thm:WaldhausenS}
  Waldhausen's $\Sdot$-construction of a
  stable $\infty$-category $\mathcal A$ is a decomposition space.
\end{theorem}
\begin{proof}
  The proof is exactly the same as in the abelian case, relying on the following
  three facts: 
  \begin{enumerate}
    \item The $\operatorname{Dec}$s are Segal spaces.
  
    \item $s_0 : S_0 \to S_1$ is a monomorphism of $\infty$-groupoids.
  
    \item A map (playing the role of monomorphisms) is an equivalence if and only if its cofibre
    is the zero object, and a map (playing the role of epimorphism) is an equivalence 
    if and only if its fibre is the zero object.
  \end{enumerate}
\end{proof}

\begin{blanko}{Remark.}
  This theorem was proved independently (and first) by Dyckerhoff and 
  Kapranov~\cite{Dyckerhoff-Kapranov:1212.3563}, Theorem 7.3.3.
  They prove it more generally for exact $\infty$-categories, a notion they
  introduce.  Their proof that Waldhausen's $\Sdot$-construction of an exact 
  $\infty$-category is a decomposition space is somewhat more complicated than ours above.
  In particular their proof of unitality (the pullback condition on
  degeneracy maps) is technical and 
  involves Quillen model structures on certain marked 
  simplicial sets \`a la Lurie~\cite{Lurie:HTT}.
  We do not wish to go into exact $\infty$-categories here, and refer
  instead the reader to \cite{Dyckerhoff-Kapranov:1212.3563}, but we wish
  to point out that our simple proof above works as
  well for exact $\infty$-categories.  This follows since the three points in
  the proof hold also for exact $\infty$-categories,
  which in turn is a consequence of 
  the definitions and basic results provided in 
  \cite[Sections~7.2 and 7.3]{Dyckerhoff-Kapranov:1212.3563}.
\end{blanko}

\begin{blanko}{Hall algebras.}\label{Hall}
  The finite-support incidence algebra of a decomposition space $X$ is
  defined in \cite[7.15]{GKT:DSIAMI-2}; see also \cite{Dyckerhoff-Kapranov:1212.3563}.
  In order for it to admit a cardinality, the
  required assumptions are, in addition to $X_1$ being locally finite, 
  that $X_0$ be finite and 
  that $X_2 \to X_1 \times X_1$ be a finite map.  In the
  case of $X= \Sdot\mathcal A$ for an abelian category $\mathcal A$, this
  translates into the condition that $\Ext^0$ and $\Ext^1$ be finite (which in
  practice means `finite dimension over a finite field').  The finite-support
  incidence algebra in this case is the {\em Hall algebra} of $\mathcal A$
  (cf.~Ringel~\cite{Ringel:Hall}; see also 
  \cite{Schiffmann:0611617},
  although these sources twist the multiplication by the so-called Euler form).

  For a stable $\infty$-category $\mathcal A$, with mapping spaces assumed to be
  locally finite (\cite[3.1]{GKT:HLA}), the finite-support incidence algebra of
  $\Sdot\mathcal A$ is the {\em derived Hall algebra}.  These were introduced by
  To\"en~\cite{Toen:0501343} in the setting of dg-categories.
  
  Hall algebras were one of the main motivations for Dyckerhoff and 
  Kapranov~\cite{Dyckerhoff-Kapranov:1212.3563} to
  introduce $2$-Segal spaces.  We refer to their work for development 
  of this important topic, recommending as entry point
  the lecture notes of Dyckerhoff~\cite{Dyckerhoff:1505.06940}.
\end{blanko}

\end{document}